\documentclass{amsart}
\usepackage{amsthm}
\usepackage{graphicx,textcomp}
\usepackage{amsmath,bm,amsfonts,mathrsfs,amssymb}
\usepackage{hyperref}
\newtheorem{theorem}{Theorem}[section]
\newtheorem{lemma}[theorem]{Lemma}
\newtheorem{prop}[theorem]{Proposition}

\newtheorem{rem}[theorem]{Remark}

\usepackage{comment} 
\usepackage{appendix}
\usepackage{cancel}
\usepackage{graphicx,color}
\usepackage{enumerate}

\usepackage{xcolor}

\begin{document}
\title{On an infinitesimal Polyakov formula  for genus zero polyhedra }
	\author{Alexey Kokotov}
	\address{Department of Mathematics \& Statistics, Concordia University, 1455 De Maisonneuve Blvd. W. Montreal, QC  H3G 1M8, Canada, \url{https://orcid.org/0000-0003-1940-0306}}
	\email{alexey.kokotov@concordia.ca}
	\thanks{The research of the first author was supported by Max Planck Institute for Mathematics in Bonn}
	
	\author{Dmitrii Korikov}
	\address{Department of Mathematics \& Statistics, Concordia University, 1455 De Maisonneuve Blvd. W. Montreal, QC  H3G 1M8, Canada, \url{https://orcid.org/0000-0002-3212-5874}}
	\email{dmitrii.v.korikov@gmail.com}
	\thanks{The research of the second author was supported by Fonds de recherche du Qu\'ebec.}
	
	
	\date{\today}
	
	\maketitle
	\begin{abstract}
		Let $X$ be a genus zero compact polyhedral surface (the Riemann sphere equipped with a flat conical metric $m$). We derive the variational formulas for the determinant of the Laplacian, ${\rm det}\,\Delta^m$, on $X$ under infinitesimal variations of the  positions of the conical points and the conical angles (i. e. infinitesimal variations of $X$ in the class of polyhedra with the same number of vertices). Besides having an independent interest, this derivation may serve as a somewhat belated mathematical counterpart of the well-known heuristic calculation of ${\rm det}\,\Delta^m$ performed by Aurell and Salomonson in the 90-s.              
	\end{abstract}
\section{Introduction}
Let $X$ be the Riemann sphere with flat conical metric $m$, having conical angles $\beta_k=2\pi(b_k+1)$, $b_k>-1$ at conical points $z_k$, $k=1, \dots, M$. 
Then one has 
$$b_1+b_2+\dots+b_M=-2$$
(due to the Gauss-Bonnet theorem) and 
$$m=C\prod_{k=1}^M|z-z_k|^{2b_k}|dz|^2$$
with some $C>0$.
Alternatively, $X$ can be introduced as a compact polyhedral  surface of genus zero, i. e. a closed genus zero surface glued from Euclidean triangles (see, e. g., \cite{Troyanov}).

Let ${\rm det}\Delta^m$ be the (modified, i. e. with zero mode excluded) $\zeta$-regularized determinant of the Friedrichs Laplacian on $X$ corresponding to the metric $m$.
This quantity was first computed by Aurell and Salomonson in \cite{AS2} via partially heuristic arguments: a closed expression for the determinant (AS formula in the sequel) through the conical angles $\beta_k$ and the positions, $z_k$, of the conical singularities was proposed. It has the form
$${\rm det}\Delta^m=C_{AS}(\beta_1, \dots, \beta_M){\rm Area}(X,m)\prod_{k\neq l}|z_k-z_l|^{\frac{b_kb_l}{6(b_k+1)}}$$
where an expression (via Hadamard type regularization of some special diverging integral) for $C_{AS}(\beta_1, \dots, \beta_N)$ can be found in \cite{AS2}, f-la (50) and \cite{AS1}, f-las (51-54). Note that the heuristic arguments of \cite{AS2} also used in \cite{AS1} in slightly different situation are, seemingly, mathematically ungrounded (see, e. g., \cite{Mazzeo} for discussion of arising subtleties).

In \cite{Kok} it was found a comparison formula   for the determinants of the Laplacians corresponding to two conformally equivalent flat conical metrics on an arbitrary compact Riemann surface of any genus (a global Polyakov type formula for two conformally equivalent polyhedra). Initially, its derivation was based on two  ideas:

 1)to make use of the BFK gluing formula from \cite{BFK} in order to smooth out the conical singularities 
 
 2)to apply the classical comparison Polyakov formula for the arising two smooth conformally equivalent metrics. 

Studying a preliminary version of \cite{Kok}, G. Carron and L. Hillairet noticed that the second part of the argument can be significantly improved:  replacing Polyakov's comparison formula by the Alvarez one and making use of explicit calculation of Spreafico of the determinant of Dirichlet Laplacian on a cone  \cite{Spreafico}, one gets the values of all the undetermined constants in the comparison formula from the preliminary version of \cite{Kok}. This improvement was incorporated in \cite{Kok}. That is why we refer to the comparison formula from \cite{Kok} as the CHS  formula. 

As an immediate consequence of the CHS formula applied to the genus zero case, one obtains an alternative expression for  
${\rm det}\Delta^m$ as
 $${\rm det}\Delta^m=C_{CHS}(\beta_1, \dots, \beta_M){\rm Area}(X,m)\prod_{k\neq l}|z_k-z_l|^{\frac{b_kb_l}{6(b_k+1)}}$$
with $C_{CHS}(\beta_1, \dots, \beta_N)$ having an explicit expression through the Barnes double zeta function.

It was observed in \cite{Kalvin} that the values  of the expressions for  $C_{AS}$ and $C_{CHS}$ (being considered separately from formulas for the determinant, at the first view, unrelated) at the angles that are rational multiples of $\pi$ can be extracted from the  literature 
(Appendix to \cite{AS1} and \cite{Dowker}) and, not surprisingly, coincide. Thus, due to a continuity argument, the heuristic AS formula follows from the CHS formula, and this observation was called in \cite{Kalvin} the first rigorous proof of the AS formula.

It seems very natural to ask whether a direct and, in a sense, better proof (not using such involved tools as the BFK and Alvarez formulas together with lengthy and hard calculations with special functions from \cite{Spreafico}, \cite{Dowker}) of the AS formula is possible. To get such a proof one has to study the dependence of the functional
${\rm det}\Delta^m$ on positions of the singularities and  conical angles.  The first attempt to do that was made in a very interesting unpublished  manuscript of Tankut Can \cite{TC}, where a variational formula for  ${\rm det}\Delta^m$  with respect to positions of conical singularities was conjectured (of course, the formula itself easily follows from Aurell-Salomonson result, the novelty was in the way to prove it). The arguments in \cite{TC} were completely heuristic and used the machinery of conformal field theory. The conjecture of Tankut Can served as the main motivation of the present work.

In the present paper, using the machinery of classical perturbation theory and the technique of the theory of elliptic equations in singularly perturbed domains,  we prove variational formulas for ${\rm det}\Delta^m$  with respect both to the positions of conical points and the conical angles (see f-las \ref{Tankut Can} and \ref{IsthereBarnes} below). It should be noted that 
variational formulas of this type for flat conical metrics were previously known only for metrics with trivial holonomy with special (and fixed) conical angles that are integer multiples of $2\pi$ (see, e. g. \cite{KokKor}). 

 The Aurell-Salomonson type formula for the determinant can be obtained from these variational formulas via straightforward integration, so, in particular, this gives the required direct and natural proof of this old result. In the subsequent paper, using a similar technique,  we are going to  study variational formulas for the determinant of the Dolbeault Laplacian (acting in a holomorphic line bundle) under infinitesimal variations (within the same conformal class) of polyhedra of higher genus. 
 
 The structure of the paper is as follows.
  
  In Section 2
 we consider a toy example (of course, well-known to experts) of a genus zero polyhedral surface : a tetrahedron with four conical singularities of angle $\pi$; in this case the determinant of the Laplacian can be easily computed by passing to the elliptic curve that covers the tetrahedron. This result is needed to fix the undetermined constant of integration in the AS-type formulas, and to serve as a reference polyhedron to get the value of the determinant ${\rm det}\,\Delta^m$ from the comparison CHS formula. The latter calculation is shown at the end of the same Section 2.  
 
 In Section 3, for the reader convenience, we illustrate the general scheme of our derivation of the infinitesimal Polyakov formulas for polyhedra, just deriving via the same method the classical Polyakov formula for a smooth metric on the Riemann sphere.
 Of course, this proof is somewhat longer than the standard one (see,e. g. \cite{Sarnak}, or \cite{Fay2}; it should be said that the methods of \cite{Fay2} play an important role in our considerations), but, probably, it may have  some independent value. 
 
  In the main Section 4 we derive the Polyakov type variational formulas for an arbitrary genus zero polyhedron. The proofs of two technical lemmas can be found in the Appendices A and B.  
 

\section{Toy model: a tetrahedron with four vertices of conical angles $\pi$}

Here we consider a toy example for the theory of polyhedral surfaces: a  tetrahedron with four vertices of conical angles $\pi$. In this case the spectrum of the Laplacian is explicitly known and  ${\rm det}\Delta^m$ can be computed with no effort.  We closely follow \cite{KokKorJPhA}, paying more attention to the arising numerical constants.  

Let $z_1, \dots, z_4\in {\mathbb C}$, introduce a flat metric $m$ on $X$ with four conical points with conical angle $\pi$ via
$$m=\frac{|dz|^2}{|z-z_1||z-z_2||z-z_3||z-z_4|}\,.$$
Consider the ramified double covering of the Riemann sphere 
with ramification points $z_1, \dots, z_4$. This is an elliptic curve $E$ equipped with flat nonsingular metric given by the 
modulus square of the holomorphic one-form
$$\omega=\frac{dz}{\sqrt{(z-z_1)(z-z_2)(z-z_3)(z-z_4)}}\,;$$
this metric coincides with the lift of the metric $m$.

Choose the basic $a$ and $b$-cycles on $E$ in the standard way, and let $A$ and $B$ be the corresponding  periods of the form $\omega$. 
Then $E$ is obtained via factorization of ${\mathbb C}$ over the lattice $L=\{mA+nB\}$ and a local coordinate on $E$ is given by
 $\zeta(P)=\int_{z_1}^P\omega$. The map $\zeta\mapsto -\zeta ({\rm mod}\,L )$ generates a holomorphic involution, $*$, of $E$ with four fixed points.

 The factorization map $E\to E/*$ coincides with the (ramified) covering map from the above. 
The nonzero eigenvalues  of the Laplacian, $\Delta^{|\omega|^2}$ on $E$ corresponding to the metric $|\omega|^2$ are double,  each nonzero eigenvalue has two eigenfunctions: one of them is $*$-invariant and another is $*$-antiinvariant.  The $*$-invariant eigenfunction descends to the eigenfunction of $\Delta^m$ corresponding to the same eigenvalue. 
This gives the relation $$\zeta_{\Delta^{|\omega|^2}}(s)=2\zeta_{\Delta^m}(s)$$ for the operator $\zeta$-functions  of $\Delta^{|\omega|^2}$ and $\Delta^m$.
In particular, one gets the equality 
\begin{equation}\label{relation}{\rm det}'{\Delta^{|\omega|^2}}=\left({\rm det}'{\Delta^m}\right)^2\end{equation}
for the determinants of the Laplacians (with zero mode excluded).  
The value of the determinant in the left hand side of (\ref{relation}) is well known and is given by
\begin{equation}
	{\rm det}'\Delta^{|\omega|^2}={\rm Area}\,(E)\,\Im(B/A)\, |\eta(B/A)|^4\,
	\end{equation}
where ${\rm Area}\,(E)=|\Im(A\bar B)|$ and $\eta$ is the Dedekind eta-function (cf., e. g., \cite{TakhMcInt}, derivation of formula (1.3), mind the extra factor $1/4$ in the definition of the Laplacian there). 

Thus,
\begin{equation}
	{\rm det}'{\Delta^m}=\frac{1}{|A|}{\rm Area}\,(E)|\eta(B/A)|^2
	\end{equation}

Using, the identity $-2\pi\eta^3(\sigma)=\theta'_1(\sigma)$,  the Jacobi identity, $\theta_1'=\pi\theta_2\theta_3\theta_4$,  for the theta-constants and the Thomae formulas for the theta-constants,
$$\theta_k^8=\frac{1}{(2\pi)^4}A^4(z_{j_1}-z_{j_2})^2(z_{j_3}-z_{j_4})^2\,,$$
where $k=2, 3, 4$ and $(j_1, j_2, j_3, j_4)$ are appropriate permutations of $(1, 2, 3, 4)$, 
one gets the relation
$$|\eta(B/A)|^2=\frac{|A|}{2^{5/3}\pi}  \prod_{i<j}|z_i-z_j|^{1/6}\,.$$
  
In addition, one has 
$${\rm Area}\,(E)=2{\rm Area}\,(X)=2\int_X\frac{|dz|^2}{|z-z_1||z-z_2||z-z_3||z-z_4|}\,$$

Thus, we obtain an explicit formula for the determinant of the Laplacian on the tetrahedron $X$:
\begin{equation}\label{tetr}
		{\rm det}'{\Delta^m}=\frac{1}{2^{2/3}\pi}\int_X\frac{|dz|^2}{|z-z_1||z-z_2||z-z_3||z-z_4|}\prod_{i<j}|z_i-z_j|^{1/6}\,.
\end{equation}

\subsection{Computation of ${\rm det}'{\Delta^m}$ via CHS formula}
As we noticed in Introduction one can derive an explicit formula for ${\rm det}'{\Delta^m}$ (an alternative to AS formula) as an immediate corollary of comparison formula (11) (Proposition 1) from \cite{Kok}. The most obvious way to do this is to make use of the following convenient form of Proposition 1 from \cite{Kok} for genus zero case (it was proposed by Tankut Can in \cite{TC}).
Let $m_1=\prod_{i=1}^N|z-P_j|^{2a_j}|dz|^2$ and $m_2=\prod_{i=1}^M|z-Q_i|^{2b_i}|dz|^2$ be two flat conical metrics on the Riemann sphere ($\sum a_j=\sum b_i=-2$). Then
\begin{equation}\label{Tankut}
\log \left[\frac{{\rm det}'\Delta^{m_1} }{ {\rm det}'\Delta^{m_2} }   \right]=\log \left[
\frac{\prod_{i=1}^NC(a_i){\rm Area}\,(X, m_1)}         {\prod_{j=1}^MC(b_j){\rm Area}\,(X, m_2)}\right]+\end{equation}
$$
+\frac{1}{6}\sum_{k<l}a_ka_l\left(\frac{1}{1+a_k}+\frac{1}{1+a_l}\right)\log|P_k-P_l|-\frac{1}{6}\sum_{k<l}b_kb_l\left(\frac{1}{1+b_k}+\frac{1}{1+b_l}\right)\log|Q_k-Q_l|\,.$$

Here the constant $C(a)$ is the ratio of two determinants: the determinant of the Laplace operator with Dirichlet boundary conditions on the right circular cone with slant height $1/(a+1)$ and the conical angle $2\pi(a+1)$ and the determinant of the Laplacian with Dirichlet boundary conditions in the unit disk. This constant is explicitly computed in \cite{Spreafico}, Theorem 1 (see also \cite{Klevtsov}, f-la (B.13) for a shorter expression via the Barnes double zeta function). 

Equation (\ref{Tankut}) can be obtained from Proposition 1 from \cite{Kok} by means of the following simple observation (due to T. Can). To compute the  quantities ${\bf g_k}$,
${\bf f_k}$ from (11) in \cite{Kok} one does not need to know explicit expressions for the distinguished local parameters $x_k$ near conical point $P_k$ (which are hard to find). It is possible to replace the distinguished local parameters $x_k$ everywhere in (11) from \cite{Kok} by  arbitrary local parameters $\zeta_k$ with  property $\zeta_k=x_k+o(x_k)$ as $x_k\to 0$. Say,
for the metric $\prod_{i=1}^N|z-P_i|^{2a_i}|dz|^2$ one can replace the distinguished local parameter $x_k$ near $P_k$ by the local parameter
$\zeta_k=\prod_{i\neq k}(P_k-P_i)^{a_i/(1+a_k)}(z-P_k)$. After this replacement  formula (11) from \cite{Kok} turns into a completely explicit one and a straightforward calculation shows that it reduces to (\ref{Tankut}). 

Choosing in (\ref{Tankut}) as $m_2$ the metric of the tetrahedron ($b_1=b_2=b_3=b_4=-1/2$) and making use of (\ref{tetr}), one immediately gets a closed explicit expression for ${\rm det}'\Delta^{m_1}$ which constitutes the claim of Proposition 3.3 from \cite{Kalvin}.

\section{Classical Polyakov formula on the Riemann sphere} 
\label{sec smooth Polyakov}

First, let us briefly deduce the classical infinitesimal Polyakov formula for the real-analytic family 
$$t\mapsto m_t=e^{-\phi_t}|dz|^2$$ 
of smooth metrics on the sphere $X=\overline{\mathbb{C}}$. This exemplifies the main steps of the reasoning that will be used in the non-smooth case. From now on, we omit the dependence on the metric $m$ and the parameter $t$ in the notation while the dot denotes the differentiation in $t$. 

\subsubsection*{\bf Variation of individual eigenvalues.} Let $t\mapsto\lambda_k(t)$ ($k=1,\dots$) be families of the nonzero eigenvalues of $\Delta_t=-4e^{\phi_t}\partial_z\partial_{\overline{z}}$ counted with their multiplicities in such a way that $\lambda_1(0)\le \lambda_2(0) \le \dots$; let also $t\mapsto \{u_k(\cdot,t)\}_{k=1,2,\dots}$ be the corresponding family of orthonormal bases of eigenfunctions. Using the standard perturbation theory, one can chose each family $t\mapsto u_k(\cdot,t)$ in such a way that $u_{k}(x,t)$ is smooth in $(x,t)$ as long as $\lambda_k(t)$ is simple, where $x$ is an arbitrary (smooth) coordinate on the sphere. 

Differentiating the equation $(\Delta-\lambda_k)u_k=0$ in $t$ and taking into account that 
$$\dot{\Delta}=\dot{\phi}\Delta,$$ 
one arrives at
\begin{equation}
\label{variation of eigenfucntions}
(\Delta-\lambda_k)\dot{u}_k=(\dot{\lambda}_k-\dot{\Delta})u_k=(\dot{\lambda}_k-\dot{\phi}\lambda_k)u_k=\lambda_k (\dot{\kappa}_k-\dot{\phi})u_k,
\end{equation}
where $\kappa_k={\rm log}\lambda_k$. Since the right-hand side of (\ref{variation of eigenfucntions}) must be orthogonal to ${\rm Ker}(\Delta-\lambda_k)\ni u_k$ in $L_2(X;g)$, we have
\begin{equation}
\label{variation of eigenvalues}
\dot{\kappa}_k=\int_X\dot{\phi}u_k^2dS,
\end{equation}
where $dS=e^{-\phi}d\overline{z}\wedge dz/2i$ is the area element. 

If $\lambda_j(0)=\dots=\lambda_{j+m-1}(0)$ is an eigenvalue $\Delta_t$ of multiplicity $m$, then the sums $\sum_{k=0}^{m-1}\kappa_{j+k}$ and $\sum_{k=0}^{m-1}u_k(x)u_k(y)$ are differentiable in $t$, and
\begin{equation}
\label{variation of eigenvalues mult}
\partial_t\Big(\sum_{k=0}^{m-1}\kappa_{j+k}\Big)=\int\limits_X\dot{\phi}\Big(\sum_{k=0}^{m-1} u_k^2\Big)dS=\int\limits_X\dot{\phi}(x)\Big[\underset{\mu=\lambda_j(0)}{\rm Res}R_\mu(x,y)\Big]\Big|_{y=x}dS(x),
\end{equation}
where $R_\mu(x,y)=R_{\mu,t}(x,y)$ is the resolvent kernel of $\Delta=\Delta_t$. Formula (\ref{variation of eigenvalues mult}) is proved in Appendix \ref{SPD theory} (where the even more complicated case of families of metrics with conical singularities is considered).

\subsubsection*{\bf Variation of $\zeta_{\Delta-\mu}(2)$.} Formulas (\ref{variation of eigenvalues}) and (\ref{variation of eigenvalues mult}) imply
$$\partial_t((\lambda_k-\mu)^{-2})=-2(\lambda_k-\mu)^{-3}\lambda_k\int_X\dot{\phi}u_k^2dS=-\int_X \partial_\mu^2\Big(\mu \frac{u_k^2}{\lambda_k-\mu}\Big)\dot{\phi}dS.$$
According to the Weyl's law $\lambda_k\sim k$, one has $\sum_{k=N}^{\infty}|\partial_t((\lambda_k-\mu)^{-2})|\le C\max\limits_X|\dot{\phi}|/N$. Making summation over $k$, one arrives at
\begin{equation}
\label{zeta of 2}
\dot{\zeta}_{\Delta-\mu}(2)=\int_X \big[-\partial_\mu^2(\mu R_\mu(x,y))\big]_{y=x}\dot{\phi}(x)dS(x),
\end{equation}
Since $\partial^2_\mu$ kills all the terms linear in $\mu$, one can replace the function in the square brackets with $\partial_\mu^2\psi_\mu$, where
\begin{equation}
\label{psi mu}
\psi_\mu(x):=-\mu \Big[R_{\mu}(x,y)+\frac{1}{A\mu}+\frac{{\rm log}d(x,y)}{2\pi}\Big]_{y=x}
\end{equation}
and $A$ and $d$ denotes the area of the sphere and the geodesic distance in the metric $m=m_t$. The well-known results on the near-diagonal asymptotics of the resolvent kernel of Laplacian (see formula (\ref{resolvent kernal asymp}) below) imply that $\psi_\mu(x)$ is finite and smooth in $x\in X$. Note that $\psi_0=0$. 

Recall that the following asymptotics 
\begin{align}
\label{resolvent kernal asymp}
R_\mu(x,y)+\frac{{\rm log}d(x,y)}{2\pi}=-\frac{1}{2\pi}\Big[\gamma+\frac{{\rm log}(4|\mu|)}{2}+\frac{K(x)}{6\mu}\Big]+\tilde{R}_\mu(x,y)
\end{align}
holds for the resolvent kernel (see Theorem 2.7, \cite{Fay2}). Here $K$ is the Gaussian curvature of the metric $g$, and the remainder $(x,y)\mapsto\tilde{R}_\mu(x,y)$ is a continuous function obeying $\tilde{R}_\mu(x,x)=O(|\mu|^{-2})$ as $\Re\mu\to-\infty$ uniformly in $x$. In particular, we have
\begin{equation}
\label{psi asymp smooth}
\psi_\mu(x)=\frac{\mu\,{\rm log}(4|\mu|)}{4\pi}+\frac{\mu\gamma}{2\pi}+\psi_\infty(x)+O(|\mu|^{-2}), \qquad \Re\mu\to-\infty,
\end{equation}
where the constant term $\psi_\infty(x)$ is given by
\begin{equation}
\label{psi infty smooth}
\psi_\infty(x)=\frac{K(x)}{12\pi}-\frac{1}{A}.
\end{equation}

\subsubsection*{\bf Variation of $\zeta_{\Delta}(s)$.} In view of the residue theorem, we have
$$(s-1)\lambda_k^{-s}=\frac{1}{2\pi i}\int_\Gamma\frac{\mu^{1-s}d\mu}{(\lambda_k-\mu)^{2}},$$
where $\Gamma$ is the contour enclosing the cut $(-\infty,0]$. Making summation over $k$, we arrive at
\begin{equation}
\label{zeta of s via zeta of 2}
(s-1)\dot{\zeta}_{\Delta}(s)=\int_\Gamma\dot{\zeta}_{\Delta-\mu}(2)\frac{\mu^{1-s}d\mu}{2\pi i}
\end{equation}
(for $\Re s\le 1$, both sides of this formula should be understood as analytic continuations of them from the half-plane $\Re s>1$).

Now, we make use of the following lemma (see Lemma 5.1, \cite{KokKorWZ}).
\begin{lemma}
\label{magiclemma}
Let $\Psi$ be a function holomorphic in some neighborhood of $(-\infty,0]$ containing the curve $\Gamma$. Suppose that the asymptotics
\begin{equation}
\label{F asymp}
\Psi(\mu)=\sum_{k=1}^K (\Psi_k+\tilde{\Psi}_k\mu{\rm log}(-\mu))\mu^{r_k}+\Phi(\mu)
\end{equation}
is valid as $\Re\mu\to -\infty$, where $r_k\in\mathbb{R}$, $\Psi_k,\tilde{\Psi_k}\in\mathbb{C}$, and $|\mu^k\partial^k_\mu\Phi(\mu)|=O(\mu^{\kappa})$ for some $\kappa<0$ and all $k=0,1,\dots$. Denote by $\Psi(\infty)$ and $\tilde{\Psi}(\infty)$ the constant term and the coefficient at ${\rm log}(-\mu)$ in {\rm(\ref{F asymp})}. 

Let $\widehat{\Psi}$ be the analytic continuation of the integral
$$\widehat{\Psi}(s):=\int\limits_{\Gamma}\partial_\mu^2\Psi(\mu)\,\frac{\mu^{1-s}d\mu}{2\pi i}$$
initially defined for sufficiently large $\Re s$. Then $\widehat{\Psi}$ is holomorphic at $s=0$ and
$$\widehat{\Psi}(0)=\tilde{\Psi}(\infty), \qquad \partial_s\widehat{\Psi}(0)=\Psi(\infty)-\tilde{\Psi}(\infty)-\Psi(0).$$

In particular, the function $s\mapsto\eta(s):=\widehat{\Psi}(s)/(s-1)$ obeys
$$\eta(0)=-\tilde{\Psi}(\infty), \qquad -\partial_s\eta(0)=\Psi(\infty)-\Psi(0).$$
\end{lemma}
For the convenience of the reader, Lemma \ref{magiclemma} is proved in Appendix A.

Let us substitute $\Psi(\mu)=\dot{\zeta}_{\Delta-\mu}(2)$ into Lemma \ref{magiclemma}. Then formulas (\ref{psi asymp smooth}), (\ref{psi infty smooth}), and (\ref{zeta of 2}) provide asymptotics (\ref{F asymp}) with 
$$\Psi(\infty)=\int_X\psi_\infty\dot{\phi}dS, \qquad \tilde{\Psi}(\infty)=0,$$  
while (\ref{zeta of s via zeta of 2}) implies $\eta(s)=\dot{\zeta}_{\Delta}(s)$. Thus, Lemma \ref{magiclemma} yields
$$\partial_t({\rm log\,det}\Delta)=\int_X[\psi_\infty-\psi_0]\dot{\phi}dS=\frac{1}{12\pi}\int_X K\dot{\phi}dS+\int_X \frac{-\dot{\phi}dS}{A}.$$
Since $d\dot{S}=-\dot{\phi}dS$, the last term in the right-hand side coincides with $\partial_t({\rm log}A)$. So, we have arrived to the infinitesimal version of the Polyakov formula
$$\partial_t{\rm log}\,({\rm det}\Delta/A)=\frac{1}{12\pi}\int_X K\dot{\phi}dS.$$

\section{Infinitesimal Polyakov's formula for polyhedral metrics on sphere}
Now, we derive the analogue of the infinitesimal Polyakov formula for the determinant of the Laplacian $\Delta=-4e^{\phi}\partial_z\partial_{\overline{z}}$ on the sphere endowed with polyhedral metric (the metric that is flat outside the finite number of conical singularities). Each such metric is of the form 
\begin{equation}
\label{polyhedral metric}
m=e^{-\phi}|dz|^2=C\prod_{j=1}^M |z-z_j|^{2b_j}|dz|^2,
\end{equation} 
where $z_k$-s are positions of the conical points while $\beta_k=2\pi(b_k+1)$ are the corresponding conical angles; then
$$\phi=-2\Re\Big(\sum_{j=1}^M b_j{\rm log}(z-z_j)\Big)-{\rm log}C.$$
For simplicity and without loss of generality, we assume that $z=\infty$ is not a conical point; then $\sum_k b_k=-2$ (the Gauss-Bonnet formula).

We consider the variations of the positions of the vertices $z_i$
\begin{equation}
\label{vert shift}
t=z_i, \qquad \dot{\phi}:=\frac{\partial\phi}{\partial z_i}=\frac{b_i}{z-z_i}
\end{equation}
or their conical angles
\begin{equation}
\label{vert opening}
t=\beta_i, \qquad \dot{\phi}:=\frac{\partial\phi}{\partial\beta_i}=\frac{1}{\pi}{\rm log}\Big|\frac{z-z_1}{z-z_i}\Big|
\end{equation}
(in the last formula, the constraint $\dot{\beta}_1=-\dot{\beta}_i=-1$ is imposed to preserve the equality $\sum_k b_k=-2$), or the overall scaling factor
\begin{equation}
\label{scaling}
t=C, \qquad \dot{\phi}:=-1/C.
\end{equation}

For sufficiently small $\epsilon>0$, denote $\mathbb{K}_j(\epsilon):=\{z\in X \ | \ d(z,z_j)\le\epsilon\}$ of conical points $z_j$. Introduce the (multi-valued) function
\begin{equation}
\label{conical coordinates}
\xi_{j}:=\sqrt{C}\int_{z_j}^z\prod_{k=1}^K (\varkappa-z_k)^{b_k}d\varkappa;
\end{equation}
then $m=|d\xi_j|^2$. Denote $r_j:=|\xi_j|$ and $\varphi_j:={\rm arg}\xi_j$. Note that the function $\zeta_j:=\xi_j^{1/(b_j+1)}$ is single-valued and is a local coordinate near $z_j$ obeying $m=|d(\zeta_j^{b_j+1})|^2$. Note that the coordinate $\zeta_j$ is ``comoving'', i.e., it depends on $t$ while $\zeta_j(z_j)=0$ for any $t$. Introduce the map $\mathscr{Z}^{(j)}_t:\,(0,\epsilon)\times\big(\mathbb{R}/\beta_j\mathbb{Z}\big)\to \mathbb{K}_j(\epsilon)$ by 
\begin{equation}
\label{loc coord map}
\mathscr{Z}^{(j)}_t(r_j,\varphi_j):=\zeta^{-1}_j(r^{2\pi/\beta_j}e^{2\pi i\varphi_j/\beta_j}).
\end{equation}
Also, put $\mathfrak{T}_{\beta',\beta}(r,\varphi+\beta\mathbb{Z}):=(r,\beta'\beta^{-1}\varphi+\beta'\mathbb{Z}).$

\subsubsection*{\bf Variation of individual eigenvalues.} 
Let $t\mapsto\lambda_k(t)$ ($k=1,\dots$) be families of the nonzero eigenvalues of $\Delta=\Delta_t$ counted with their multiplicities in such a way that $\lambda_1(0)\le \lambda_2(0) \le \dots$; let also $t\mapsto \{u_k(\cdot,t)\}_{k=1,2,\dots}$ be the corresponding family of orthonormal bases of eigenfunctions. First, we prove that the family of $t\mapsto(\lambda_k(t),u_k(\cdot,t)$ is  differentiable in $t$ as long as $\lambda_k(t)$ is simple. To this end, we apply the technique of the theory of elliptic problems in singularly perturbed domains (see Chapters 4 and 6, \cite{MNP}).
\begin{lemma}
\label{SPD lemma}
Suppose that the eigenvalue $\lambda(t)$ of $\Delta_t$ is simple for all $t\in(-t_0,t_0)$. Then the following statements hold:
\smallskip
\begin{enumerate}
\item the corresponding family of normalized eigenfunction $(x,t)\mapsto u(x,t)$ can be chosen to be smooth outside the vertices $(z_k(t),t)$, where $x$ is an arbitrary smooth coordinate on the sphere. As a corollary, differentiating the equation $(\Delta-\lambda_k)u=0$ in $t=z_i,\beta_i$, one shows that equation {\rm(\ref{variation of eigenfucntions})} still holds outside vertices.
\item For each $t$, the solution $u=u(\cdot,t)$ to $(\Delta-\lambda)u=0$ admits the expansion into convergent series
\begin{equation}
\label{eigenfunc expansion near vert}
\begin{split}
u\circ\mathscr{Z}^{(j)}_t(r,\varphi)=\frac{1}{\beta\epsilon}\sum_{k\in\mathbb{Z}}\frac{J_{2\pi|k|/\beta}(r\sqrt{\lambda})}{J_{2\pi|k|/\beta}(\epsilon\sqrt{\lambda})}e^{2\pi k i\varphi/\beta}(f,e^{2\pi k i\varphi/\beta})_{L_2(\partial \mathbb{K}(\epsilon))}=\\
=\sum_{k\in\mathbb{Z}}\sum_{j=0}^{\infty}c_{j}\Big(\frac{2\pi k}{\beta}\Big)\,(f,e^{2\pi k i\varphi/\beta})_{L_2(\partial \mathbb{K}(\epsilon))}(r\sqrt{\lambda})^{2j+2\pi |k|/\beta}e^{2\pi k i\varphi/\beta}
\end{split}
\end{equation}
in $\mathbb{K}_j(\epsilon)$, where $J_\cdot$ is the Bessel function, $\beta=\beta_j$, and $f(\varphi)=u\circ\mathscr{Z}^{(j)}_t(\epsilon,\varphi)$. The coefficients $c_{j}(\nu)$ $(j>0)$ in {\rm(\ref{eigenfunc expansion near vert})} and their derivatives with respect to $\nu$ decay super-exponentially as $j\to+\infty$ or $\nu\to+\infty$. Thus, series {\rm(\ref{eigenfunc expansion near vert})} admit term-wise differentiation in $t,r,\varphi$. In addition, 
\begin{align}
\label{dot u near vertex est}
\begin{array}{lr}
\dot{u}\circ\mathscr{Z}^{(j)}_t=O(1), \ \frac{\partial\dot{u}}{\partial r}\circ\mathscr{Z}^{(j)}_t=\sum\limits_{\pm}c_\pm \frac{e^{\pm 2\pi i\varphi/\beta}}{r}+O(r^{\delta-1}) & (t=z_j),\\
\dot{u}\circ\mathscr{Z}^{(j)}_t=O(r^{\frac{2\pi}{\beta}}{\rm log}r), \ \frac{\partial\dot{u}}{\partial r}\circ\mathscr{Z}^{(j)}_t=O(r^{\frac{2\pi}{\beta}-1}{\rm log}r) & (t=\beta_j),
\end{array}
\end{align}
as $r\to 0$, where $\delta>0$.
\end{enumerate}
\end{lemma}
To make the exposition self-contained, the sketch of the proof of Lemma \ref{SPD lemma} based on the usual perturbation theory is presented in Apppendix \ref{SPD theory}.

Let $X(\epsilon)$ be the complement of all $\mathbb{K}_j(\epsilon)$ in $X$. Multiplying both parts of equation (\ref{variation of eigenfucntions}) (provided by Lemma \ref{SPD lemma}, 1.) by $u_k$, integrating over $X(\epsilon)$, one arrives at
\begin{align}
\label{integrat b p limit}
\begin{split}
\lambda_k\lim_{\epsilon\to 0}\int_{X(\epsilon)}(\dot{\kappa}_k-\dot{\phi})|u_k|^2dS=&\lim_{\epsilon\to 0}\int_{X(\epsilon)}(\Delta-\lambda_k)\dot{u}_k u_k dS=\\
=&\lim_{\epsilon\to 0}\int_{\partial X(\epsilon)}[\partial_\nu\dot{u}u-\dot{u}\partial_\nu u]dl=0.
\end{split}
\end{align}
Indeed, if $t=z_j$, then the last integral in (\ref{integrat b p limit}) is equal to
$$\int\limits_{0}^{\beta_j}
\Big(\sum_{\pm}c_\pm e^{\frac{\pm 2\pi i\varphi}{\beta_j}}\Big)u(z_j)d\varphi_j+o(1)=0+o(1)$$
due to (\ref{dot u near vertex est}) and (\ref{eigenfunc expansion near vert}). The same fact for $t=\beta_j$ is obtained even more simply from (\ref{dot u near vertex est}), (\ref{eigenfunc expansion near vert}) and (\ref{integrat b p limit}).

Therefore, in the polyhedral case, one again arrives to (\ref{variation of eigenvalues}), where the principal value should be taken in the right-hand side, 
\begin{equation}
\label{variation of eigenvalues conic}
\dot{\kappa}_k={\rm p.v.}\int_X\dot{\phi}u_k^2dS,
\end{equation}
Indeed, the asymptotics of the integrand in (\ref{variation of eigenvalues}) at the vertex $z=z_i$ is given by
$$\dot{\phi}u_k^2dS=b_i u_k^2[(z-z_i)^{-1}+O(1)]r_idr_id\varphi_i=c[\xi_i^{-\frac{2\pi}{\beta_i}}J_{0}(r_i\sqrt{\lambda_k})^2+O(1)]r_idr_id\varphi_i$$
where the singular term $c\xi_i^{-\frac{2\pi}{\beta_i}}J_{0}(r_i\sqrt{\lambda_k})^2 r_i$ is killed after the integration in $\varphi_i$. Therefore, formula (\ref{variation of eigenvalues}) remains valid after taking the principal value in the right-hand side.

Now, consider the case in which $\lambda_k(0)=\dots=\lambda_{k+m-1}(0)=:\lambda(0)$. For this case, we prove the formula
\begin{equation}
\label{variation of eigenvalues conic multip}
\begin{split}
\partial_t\Big(\sum_{j=0}^{m-1}\kappa_{k+j}\Big)\Big|_{t=0}=&\mathscr{I}\,{\rm p.v.}\int_X\dot{\phi}\Big(\sum_{j=0}^{m-1}u_{k+j}^2\Big)dS=\\
&=\mathscr{I}\,{\rm p.v.}\int_X\dot{\phi}\,\,\underset{\mu=\lambda(0)}{\rm Res}R_\mu(x,y)\Big|_{y=x}dS.
\end{split}
\end{equation}
Here $\mathscr{I}$ in the right-hand side is the operator eliminating removable discontinuities, $\mathscr{I}f(t)=\lim\limits_{\tau\to+0}\int_{t-\tau/2}^{t+\tau/2}f(t')dt'$.
\begin{lemma}
\label{SPD lemma mult}
Let $\lambda_k,\dots,\lambda_{k+m-1}$ be all the eigenvalues of $\Delta$ taking the value $\lambda(0)$ for $t=0$. Then
\begin{enumerate}
\item For any function $\mathcal{E}$ holomorphic near $\lambda(0)$ and for sufficiently small $t\in(-t_0,t_0)$, the function
\begin{equation}
\label{F residue func}
(x,y,t)\mapsto\mathcal{F}_t(x,y|\mathcal{E}):=\sum_{j=0}^{m-1}\mathcal{E}(\lambda_{k+j}(t))u_{k+j}(x)u_{k+j}(y)
\end{equation}
is smooth outside the vertices $(z_k(t),z_{k'}(t),t)$. 
\item In particular, the function $$\sum_{j=0}^{m-1}\kappa_{k+j}=\int_{X}\mathcal{F}_t(x,x|1)dS$$ is differentiable in $t\in(-t_0,t_0)$ and formula {\rm(\ref{variation of eigenvalues conic multip})} is valid.
\end{enumerate}
\end{lemma}
To make the exposition self-contained, the proof of Lemma \ref{SPD lemma mult} is presented in Apppendix \ref{SPD theory}.

\subsubsection*{\bf Variation of ${\rm log\,det}\Delta$: preliminary formulas.} Repeating formally the arguments of Section \ref{sec smooth Polyakov} (involving Lemma \ref{magiclemma}), one obtains
\begin{equation}
\label{det formula}
\partial_t{\rm log\,det}\Delta=\Psi_\infty \qquad (t=z_i,\beta_i),
\end{equation}
where $\Psi_\infty$ is the constant term in the asymptotics, as $\Re\mu\to -\infty$, of the integral
\begin{equation}
\label{psi int}
\Psi_\mu:={\rm p.v.}\int_X\dot{\phi}\psi_\mu dS.
\end{equation}
Here $\dot{\phi}$ and $\psi_\mu$ are given by (\ref{vert shift}), (\ref{vert opening}) and (\ref{psi mu}), respectively. 

For metrics with conical singularities, the justification of (\ref{det formula}) is much more complicated since the estimates of the regularized integral on the right-hand side of (\ref{variation of eigenvalues conic}) requires not only the $L_2(X,m)$-boundedness of eigenfunctions (which is nothing more than the normalization conditions) but also their asymptotics near the vertices with estimates of the coefficients and the remainder uniform in $k$. Since asymptotics (\ref{Psi j variation z asymp fin}), (\ref{Psi i variation z asymp fin}) required for this are also needed for calculation of the right-hand side of (\ref{det formula}), we hold over the justification to Section \ref{Justifisection}.

Let us represent $\psi_\mu$ as $\psi_\mu=\psi_\mu^{(0)}+\psi_\mu^{(1)}$, where
\begin{align}
\label{modified psi mu}
\psi_\mu^{(0)}(x)&:=\mu \Big[\frac{1}{2\pi}K_0(d(x,y)\sqrt{-\mu}))-R_{\mu}(x,y)\Big]_{y=x}, \\ \nonumber
\psi_\mu^{(1)}(x)&:=-\mu \Big[\frac{K_0(d(x,y)\sqrt{-\mu})+{\rm log}d(x,y)}{2\pi})+\frac{1}{A\mu}\Big]_{y=x}=\\ \nonumber
&\qquad\qquad\qquad\qquad\qquad=\frac{\mu\,{\rm log}(4|\mu|)}{4\pi}+\frac{\mu\gamma}{2\pi}-\frac{1}{A},
\end{align}
and $K_0$ is the Macdonald function. Now, formula (\ref{psi int}) reads
\begin{equation}
\label{psi int 1}
\Psi_\mu:=-\frac{\dot{A}\mu\,{\rm log}(4|\mu|)}{4\pi}-\frac{\dot{A}\mu\gamma}{2\pi}+\frac{\dot{A}}{A}+{\rm p.v.}\int_X\dot{\phi}\psi_\mu^{(0)}(x) dS.
\end{equation}
For polyhedral metrics (\ref{polyhedral metric}), formula (\ref{resolvent kernal asymp}) remains valid (with $K=0$) and uniform in $x,y$ outside small neighborhoods of vertices and can be specified as follows
$$R_\mu(x,y)-\frac{1}{2\pi}K_0(d(x,y)\sqrt{-\mu}))=O(e^{\epsilon_0\Re\mu}),$$
where $\epsilon_0>0$. Therefore, equality (\ref{psi int 1}) remains valid, up to the terms exponentially decaying as $\Re\mu\to-\infty$, if one replaces the domain of integration in the right-hand side of (\ref{psi int 1}) with an arbitrarily small neighborhood of vertices. Now combining formulas (\ref{det formula}) and (\ref{psi int 1}) yields
\begin{equation}
\label{det formula 1}
\partial_t{\rm log}\,({\rm det}\Delta/A)=\sum_{j=1}^{M}\Psi^{(j)}_\infty, 
\end{equation}
where $\Psi^{(j)}_\infty$ is the constant term in the asymptotics, as $\Re\mu\to -\infty$, of the integral
\begin{equation}
\label{psi vertex int}
\Psi^{(j)}_\mu:={\rm p.v.}\int_{\mathbb{K}_j(\epsilon)}\dot{\phi}\psi_\mu^{(0)} dS
\end{equation}
over the small neighborhood of the vertex $z_j$. To justify (\ref{det formula 1}), one needs to prove that each $\Psi^{(j)}_\mu$ admits asymptotics (\ref{F asymp}). To this end (and also to calculate the terms $\Psi^{(j)}_\infty$), one requires the asymptotics of the resolvent kernel $R_\mu(x,y)$ as $\Re\mu\to-\infty$ which is uniform near vertices.

\subsubsection*{\bf Parametrix for the resolvent kernel $R_\mu(x,y)$ in $\mathbb{K}_j(\epsilon)$.} For the infinite cone $\mathbb{K}$ of opening $\beta$, the heat kernel is given by (see \cite{Carslaw,Dowker})
\begin{equation}
\label{HK Cone}
H_t(r,\varphi,r',\varphi'|\beta)=\frac{1}{8\pi i\beta t}\int\limits_{\mathscr{C}}{\rm exp}\Big(-\frac{\mathfrak{r}^2}{4t}\Big)\,\Xi d\vartheta,
\end{equation}
where
\begin{equation}
\label{rtheta}
\mathfrak{r}^2:=r^2-2rr'{\rm cos}\vartheta+r'^{2}, \qquad \Xi:={\rm cot}\Big(\pi\beta^{-1}(\vartheta+\varphi-\varphi')\Big),
\end{equation}
$(r,\varphi)$ and $r',\varphi')$ are polar coordinates of the points $z$ and $z'$ of $\mathbb{K}$, respectively. The integration contour $\mathscr{C}$ is the union of the lines $\pm l:=\{\vartheta=\pm(\pi-i\acute{\vartheta})\}_{\acute{\vartheta}\in\mathbb{R}}$ and infinitesimal anti-clockwise circles $\odot[\vartheta_*]$ centered at the roots $\vartheta_*$ of $\Xi$ lying in the strip $\Re\theta\in(-\pi,\pi)$.   

Separating the contribution of the pole at $\vartheta=\varphi'-\varphi$ and assuming that $x$ and $x'$ are close enough, one rewrites (\ref{HK Cone}) as 
\begin{equation}
\label{HK Cone 1}
H_t(r,\varphi,r',\varphi'|\beta)=\frac{1}{4\pi t}{\rm exp}\Big(-\frac{d(z,z')^2}{4t}\Big)+\frac{1}{8\pi i\beta t}\int\limits_{\tilde{\mathscr{C}}}{\rm exp}\Big(-\frac{\mathfrak{r}^2}{4t}\Big)\,\Xi d\vartheta,
\end{equation}
where $\tilde{\mathscr{C}}:=\mathscr{C}\backslash\odot[\varphi'-\varphi]$, the first term is just a heat kernel on the plane.

The resolvent kernel in $\mathbb{K}$ (corresponding to the non-negative Laplacian) is obtained by the Laplace transform of (\ref{HK Cone 1}),
\begin{align}
\label{RK param}
\begin{split}
R_\mu(r,\varphi,r',\varphi'|\beta)=\int\limits_{0}^{+\infty}&e^{\mu t}H^0_t(r,\varphi,r',\varphi'|\beta)dt=\\
=\frac{1}{2\pi}K_0(d(z,z')&\sqrt{-\mu})+\frac{1}{8\pi i\beta}\int\limits_{\tilde{\mathscr{C}}}d\vartheta\,\Xi \int\limits_{0}^{+\infty}{\rm exp}\Big(\mu t-\frac{\mathfrak{r}^2}{4t}\Big)\frac{dt}{t}.
\end{split}
\end{align}
Note that the right-hand side and all its derivatives decay faster than any power of $|\mu|$ as $\Re\mu\to -\infty$ uniformly in $z$ and $z'$ separated from each other.

Now, let $z\in X$ and $z'\in\mathbb{K}_j(\epsilon)$. Denote by $\chi_j$ the cut-off function equal to $1$ in a neighborhood of $\mathbb{K}_j(\epsilon)$ and introduce the parametrix 
$$\chi_j(z)R_\mu(r_j,\varphi_j,r'_j,\varphi'_j|\beta_j)$$ for the resolvent kernel $R_\mu(z,z')$ on $(X,g)$. Denote 
$$\tilde{R}_\mu(\cdot,z')=R_\mu(\cdot,z')-\chi_j R_\mu(\cdot,r'_j,\varphi'_j|\beta_j),$$
then $(\Delta-\mu)\tilde{R}_\mu(\cdot,z')=-[\Delta,\chi_j]R_\mu(\cdot,r'_j,\varphi'_j|\beta_j)$. Here the right-hand side and all its derivatives are $O(|\mu|^{-\infty})$ as $\Re\mu\to -\infty$ since $z\in {\rm supp}[\Delta,\chi_j]$ and $z'\in\mathbb{K}_j(\epsilon)$ are always separated from each other. In view of the standard operator estimate 
$$(\Delta-\mu)^{-1}=O(|\mu|^{-1}), \qquad \Re\mu\to -\infty,$$ 
the $L_2(X;g)$-norm of $\tilde{R}_\mu(\cdot,z')$ (therefore, the $L_2(X;g)$-norm of any $\Delta^l\tilde{R}_\mu(\cdot,z')$, $l=1,2,\dots$) is $O(|\mu|^{-\infty})$. In view of the smoothness increasing theorems for solutions to elliptic equations, this means that 
\begin{equation}
\label{est rem RK}
\tilde{R}_\mu(z,z')=O(|\mu|^{-\infty}), \quad \partial_z\tilde{R}_\mu(z,z')=O(|\mu|^{-\infty})
\end{equation}
uniformly in $z\in X(\epsilon)$ and $z'\in\mathbb{K}_j(\epsilon)$. To prove (\ref{est rem RK}) for $z$ close to vertices, it remains to note that $(\Delta-\mu)\tilde{R}_\mu(\cdot,z')=0$ in $\mathbb{K}_j(\epsilon)$ and, thus, representation (\ref{eigenfunc expansion near vert}) is valid for $u=\tilde{R}_\mu(z,z')$ and $\lambda=\mu$. Now the substitution of (\ref{est rem RK}) into (\ref{eigenfunc expansion near vert}) yields (\ref{est rem RK}) for any $z\in K_j(\epsilon)$. 

In view of (\ref{modified psi mu}), (\ref{RK param}), and (\ref{est rem RK}), we obtain the expansion
\begin{equation}
\label{modified psi mu asymp}
\psi_\mu^{(0)}=a_\mu(r_j)+\tilde{\psi}_\mu^{(0)}
\end{equation}
in $\mathbb{K}_j(\epsilon)$, where
\begin{align}
\label{modified psi mu main term}
a_\mu(r_j):=\frac{-\mu}{8\pi i\beta_j}\int\limits_{\tilde{\mathscr{C}}}d\vartheta\,{\rm cot}\Big(\frac{\pi\vartheta}{\beta_j}\Big) \int\limits_{0}^{+\infty}{\rm exp}\Big(\mu t-\frac{r_j^2{\rm sin}^2(\vartheta/2)}{t}\Big)\frac{dt}{t},
\end{align}
and the remainder $\tilde{\psi}_\mu^{(0)}$ obey the (uniform in $z\in\mathbb{K}_j(\epsilon)$) estimate
\begin{equation}
\label{modified psi mu remainder est 1}
\tilde{\psi}_\mu^{(0)}(z)=O(|\mu|^{-\infty}), \qquad \partial_z\tilde{\psi}_\mu^{(0)}(z)=O(|\mu|^{-\infty}).
\end{equation}
In view of (\ref{modified psi mu remainder est 1}), the function $\tilde{\psi}_\mu^{(0)}(z)=\tilde{\psi}_\mu^{(0)}(z_j)+\int_{z_j}^z\partial_{z'}\tilde{\psi}_\mu^{(0)}(z')dz'$ obeys
\begin{equation}
\label{modified psi mu remainder est 2}
\tilde{\psi}_\mu^{(0)}(z)=\tilde{\psi}_\mu^{(0)}(0)+O(|\mu|^{-\infty}(z-z_j)), \qquad z\in\mathbb{K}_j(\epsilon).
\end{equation}
Note that main term (\ref{modified psi mu main term}) is rotationally symmetric, i.e., is independent of the polar angle $\varphi_j$.

\subsection{Derivation of $\partial{\rm log}\,({\rm det}\Delta/A)/\partial z_i$.} Let $t=z_i$ and $j\ne i$. Since $(z-z_i)^{-1}$ is holomorphic near $z=z_j$, it admits the expansion
\begin{equation}
\label{Taylor i ne j}
(z-z_i)^{-1}=(z_j-z_i)^{-1}+\sum_{k=1}^{\infty}c_{jk}\xi_{j}^{\frac{2\pi k}{\beta_j}}.
\end{equation}
Substituting (\ref{vert shift}), (\ref{modified psi mu asymp}) into (\ref{psi vertex int}) and applying formulas (\ref{modified psi mu remainder est 1}), (\ref{Taylor i ne j}), one obtains
\begin{equation}
\label{Psi j variation z asymp}
\Psi^{(j)}_\mu=\frac{b_i}{z_j-z_i}\int\limits_{\mathbb{K}_j(\epsilon)}a_\mu dS+b_i\int\limits_{0}^\epsilon a_\mu r_j\,dr_j\int\limits_0^{\beta_j}d\varphi_j\,\sum_{k=1}^{\infty}c_{jk}\xi_{j}^{\frac{2\pi k}{\beta_j}}+O(|\mu|^{-\infty}).
\end{equation}
Since $a_\mu$ is rotationally symmetric while each integral $\int_0^{\beta_j}\xi_j^{\frac{2\pi k}{\beta_j}}d\varphi_j$ is zero for any nonzero $k$, the second term in the right-hand side of (\ref{Psi j variation z asymp}) is zero. 

Let us derive the asymptotics of the integral $\int\limits_{\mathbb{K}_j(\epsilon)}a_\mu dS$ as $\Re\mu\to-\infty$. In view of (\ref{modified psi mu main term}), we have
\begin{align*}
\int\limits_{\mathbb{K}_j(\epsilon)}a_\mu dS=\frac{-\mu}{8\pi i}\int\limits_{0}^{+\infty}\frac{e^{\mu t}dt}{t}\int\limits_{\tilde{\mathscr{C}}}d\vartheta\,{\rm cot}\Big(\frac{\pi\vartheta}{\beta_j}\Big)\int\limits_{0}^\epsilon{\rm exp}\Big(-\frac{r_j^2{\rm sin}^2(\vartheta/2)}{t}\Big) r_j\,dr_j.
\end{align*}
Introducing the new variables
\begin{equation}
\label{new variable}
p:=\frac{r_j^2{\rm sin}^2(\vartheta/2)}{t}, \qquad P:=\frac{\epsilon^2{\rm sin}^2(\vartheta/2)}{t}
\end{equation}
one rewrites the last formula as
\begin{align*}
\int\limits_{\mathbb{K}_j(\epsilon)}a_\mu dS=\frac{-\mu}{16\pi i}\int\limits_{0}^{+\infty}e^{\mu t}dt\int\limits_{\tilde{\mathscr{C}}}d\vartheta\frac{{\rm cot}(\pi\vartheta/\beta_j)}{{\rm sin}^2(\vartheta/2)}\int\limits_{0}^P e^{-p}dp=\\
=\frac{-\mu}{16\pi i}\int\limits_{0}^{+\infty}e^{\mu t}dt\int\limits_{\tilde{\mathscr{C}}}d\vartheta\frac{{\rm cot}(\pi\vartheta/\beta_j)}{{\rm sin}^2(\vartheta/2)}(1-e^{-P}).
\end{align*}
Note that $\Re{\rm sin}(\vartheta/2)>{\rm const}>0$ on $\tilde{\mathscr{C}}$ and grows exponentially as $\tilde{\mathscr{C}}\ni\vartheta\to\infty$. Then $e^{-P}$ and all its derivatives decay exponentially and uniformly in $\vartheta\in\tilde{\mathscr{C}}$ as $t\to +0$. Thus, the function
$$t\mapsto\int\limits_{\tilde{\mathscr{C}}}d\vartheta\frac{{\rm cot}(\pi\vartheta/\beta_j)}{{\rm sin}^2(\vartheta/2)}e^{-P}$$
can be smoothly extended by zero to the semi-axis $t<0$. Then the multiple integration by parts yields
$$\int\limits_{0}^{+\infty}e^{\mu t}dt\int\limits_{\tilde{\mathscr{C}}}d\vartheta\frac{{\rm cot}(\pi\vartheta/\beta_j)}{{\rm sin}^2(\vartheta/2)}e^{-P}=O(|\mu|^{-\infty}).$$
Therefore, one arrives at
\begin{align}
\label{integrated asymptotic const term}
\int\limits_{\mathbb{K}_j(\epsilon)}a_\mu dS=\mathfrak{Q}(\beta_j)+O(|\mu|^{-\infty}) \qquad (j=1,\dots,M).
\end{align}
where
\begin{equation}
\label{Q integral}
\mathfrak{Q}(\beta):=\frac{1}{16\pi i}\int\limits_{\tilde{\mathscr{C}}}\frac{{\rm cot}(\pi\vartheta/\beta)}{{\rm sin}^2(\vartheta/2)}d\vartheta=-\frac{1}{12}\Big(\frac{\beta}{2\pi}-\frac{2\pi}{\beta}\Big).
\end{equation}
In view of (\ref{integrated asymptotic const term}), asymptotics (\ref{Psi j variation z asymp}) takes the form
\begin{equation}
\label{Psi j variation z asymp fin}
\Psi^{(j)}_\mu=-\frac{b_i}{z_j-z_i}\Big(\frac{\beta_j}{2\pi}-\frac{2\pi}{\beta_j}\Big)\frac{1}{12}+O(|\mu|^{-\infty}) \qquad (j\ne i).
\end{equation}

Now, let $j=i$. Then (\ref{conical coordinates}) implies
\begin{align}
\label{Taylor i}
\begin{split}
\xi_{i}=\sqrt{C}\prod_{k\ne i}(z_i-z_k)^{b_k}\int\limits_{0}^{z-z_i}\big[1+\sum_{k\ne i}\frac{b_k\zeta}{z_i-z_k}+O(\zeta^2)\big]\zeta^{b_i}d\zeta=\\
=B_i(z-z_i)^{b_i+1}\big[1+\frac{\beta_iA_i}{2\pi}(z-z_i)+O((z-z_i)^2)\big],
\end{split}
\end{align}
where
\begin{equation}
\label{Taylir coefficients i}
A_i:=\sum_{k\ne i}\frac{b_k}{z_i-z_k}\frac{1}{b_i+2}, \qquad B_i:=\frac{2\pi\sqrt{C}}{\beta_i}\prod_{k\ne i}(z_i-z_k)^{b_k}.
\end{equation}
Formula (\ref{Taylor i}) implies
\begin{equation}
\label{zeta xi}
(z-z_i)^{-1}=\Big(\frac{\xi_{i}}{B_i}\Big)^{-\frac{2\pi}{\beta_i}}+A_i+\sum_{k=1}^{\infty}c_{ik}\xi_{i}^{\frac{2\pi}{\beta_i k}}.
\end{equation}
Now, the substitution of (\ref{vert shift}), (\ref{modified psi mu asymp}) into (\ref{psi vertex int}) and taking into account (\ref{zeta xi}), (\ref{modified psi mu main term}), and (\ref{modified psi mu remainder est 2}) yields
\begin{align*}
\Psi^{(j)}_\mu:={\rm p.v.}\int_{\mathbb{K}_j(\epsilon)}\frac{b_i}{z-z_i}(a_\mu+\tilde{\psi}_\mu^{(0)}(0)+O(|\mu|^{-\infty}(z-z_i)) dS=\\
=\cdot\,{\rm p.v.}\int_{\mathbb{K}_j(\epsilon)}b_i(z-z_i)^{-1}a_\mu dS+0+O(|\mu|^{-\infty})=b_i A_i\int_{\mathbb{K}_j(\epsilon)}a_\mu dS+\\
+{\rm p.v.}\int_0^{\epsilon}dr_i\,r_i b_i a_\mu\int_{0}^{\beta_i}\Big(\Big(\frac{\xi_{i}}{B_i}\Big)^{-\frac{2\pi}{\beta_i}}+\sum_{k=1}^{\infty}c_{ik}\xi_{i}^{\frac{2\pi}{\beta_i k}}\Big)d\varphi_i+O(|\mu|^{-\infty}).
\end{align*}
Here the last integral in the right-hand side vanishes since $a_\mu$ is rotationally symmetric. Now, taking into account (\ref{integrated asymptotic const term}) and (\ref{Taylir coefficients i}), one arrives at
\begin{equation}
\label{Psi i variation z asymp fin}
\Psi^{(i)}_\mu=\Big(\sum_{k\ne i}\frac{b_k}{z_k-z_i}\Big)\frac{b_i}{b_i+2}\Big(\frac{\beta_i}{2\pi}-\frac{2\pi}{\beta_i}\Big)\frac{1}{12}+O(|\mu|^{-\infty}).
\end{equation}
Substituting the constant terms in asymptotics (\ref{Psi j variation z asymp fin}), (\ref{Psi i variation z asymp fin}) into (\ref{det formula 1}) and taking into account that $\beta_j/2\pi=b_j+1$ (hence, $\beta_j/2\pi-2\pi/\beta_j=b_j(b_j+2)/b_j+1$), one obtains
\begin{align*}
12\frac{\partial {\rm log}\,({\rm det}\Delta/A)}{\partial z_i}=\sum_{j\ne i}\frac{b_i}{z_j-z_i}\Big[\frac{b_j(b_j+2)}{b_j+1}-\frac{b_j}{b_i+2}\frac{b_i(b_i+2)}{b_i+1}\Big]=\\
=\sum_{j\ne i}\frac{b_i b_j}{z_j-z_i}\Big[\frac{b_j+2}{b_j+1}-\frac{b_i}{b_i+1}\Big]=\sum_{j\ne i}\frac{b_i b_j}{z_i-z_j}\Big[\frac{2\pi}{\beta_j}+\frac{2\pi}{\beta_i}\Big].
\end{align*}
Thus, we have arrived to the Tankut Can formula
\begin{equation}
\label{Tankut Can}
\frac{\partial {\rm log}\,({\rm det}\Delta/A)}{\partial z_i}=\frac{\pi}{6}\sum_{j\ne i}\frac{b_i b_j}{z_i-z_j}\Big[\frac{1}{\beta_i}+\frac{1}{\beta_j}\Big]=\frac{\partial\mathfrak{W}}{\partial z_i},
\end{equation}
where the function $\mathfrak{W}$ is given by 
\begin{equation}
\label{DoubleYouFunc}
\mathfrak{W}(z_1,\dots,z_M,\beta_1,\dots,\beta_M):=\frac{\pi}{3}\sum_{k<l}b_k b_l\Big[\frac{1}{\beta_k}+\frac{1}{\beta_l}\Big]{\rm log}|z_k-z_l|.
\end{equation}

\subsection{Derivation of $\partial {\rm log}\,({\rm det}\Delta/A)/\partial\beta_i$} 
Let $t=\beta_i$ and $j\ne 1,i$. Since $\dot{\phi}$ (given by (\ref{vert opening})) is harmonic in $X\backslash\{z_1,z_i\}$, it admits the expansion
$$\dot{\phi}(z)-\dot{\phi}(z_j)=\sum_{k=1}^{\infty} c_k\xi_j^k+\sum_{k=1}^{\infty} d_k\overline{\xi}_j^k$$
near $z=z_j$. Since the integration of the right-hand side multiplied by any rotationally symmetric function (such as the parametrix $a_\mu$ for $\psi_\mu^{(0)}$ given by (\ref{modified psi mu main term})) over $\mathbb{K}_j$ gives zero, formulas (\ref{psi vertex int}), (\ref{modified psi mu asymp}), (\ref{modified psi mu remainder est 1}), and (\ref{integrated asymptotic const term}) imply
\begin{equation}
\label{Psi j variation beta asymp}
\begin{split}
\Psi^{(j)}_\mu=\dot{\phi}(z_j)\int_{\mathbb{K}_j(\epsilon)}a_\mu dS+O(|\mu|^{-\infty})=&\\
=-\frac{1}{12\pi}\Big(\frac{\beta_j}{2\pi}-\frac{2\pi}{\beta_j}\Big){\rm log}\Big|\frac{z_j-z_1}{z_j-z_i}&\Big|+O(|\mu|^{-\infty})  \qquad (j\ne 1,i).
\end{split}
\end{equation}

Now, let $j=i$. Then formulas (\ref{vert opening}), (\ref{zeta xi}) yield the expansion
\begin{align}
\label{perturbation beta near z i}
\dot{\phi}=-\frac{2{\rm log}|\xi_i|}{\beta_i}+\Big(\frac{{\rm log}|z_i-z_1|}{\pi}+\frac{2{\rm log}|B_i|}{\beta_i}\Big)+\sum_{k=1}^{\infty} \tilde{c}_k\xi_i^k+\sum_{k=1}^{\infty} \tilde{d}_k\overline{\xi}_i^k
\end{align}
near $z=z_i$, where $B_i$ is given by (\ref{Taylir coefficients i}). The substitution of the last formula and (\ref{modified psi mu asymp}) into (\ref{psi vertex int}) and taking into account (\ref{modified psi mu remainder est 1}) and (\ref{integrated asymptotic const term}) yields 
\begin{align}
\label{psi mu j beta asymp preliminary}
\begin{split}
\Psi^{(i)}_\mu=\Big(\frac{{\rm log}|z_i-z_1|}{\pi}+\frac{2{\rm log}|B_i|}{\beta_i}\Big)\int\limits_{\mathbb{K}_i(\epsilon)}a_\mu dS-\frac{2}{\beta_i}\int\limits_{\mathbb{K}_i(\epsilon)}a_\mu {\rm log}|\xi_i| dS+\\
+O(|\mu|^{-\infty})=-\frac{1}{12}\Big(\frac{\beta_i}{2\pi}-\frac{2\pi}{\beta_i}\Big)\Big(\frac{{\rm log}|z_i-z_1|}{\pi}+\frac{2{\rm log}|B_i|}{\beta_i}\Big)-\\
-2\int\limits_{0}^{\epsilon}a_\mu(r_i)r_i{\rm log}(r_i) dr_i+O(|\mu|^{-\infty}).
\end{split}
\end{align}
Let us derive the asymptotics of the last term in (\ref{psi mu j beta asymp preliminary}). In view of (\ref{modified psi mu main term}), one has
\begin{align}
\label{integrated asymptotics log preliminary}
\begin{split}
\int\limits_{0}^{\epsilon}a_\mu(r_i)r_i{\rm log}(r_i) dr_i=\int\limits_{0}^{+\infty}\frac{-\mu e^{\mu t}\,dt}{8\pi i\beta_i t}\int\limits_{\tilde{\mathscr{C}}}d\vartheta\,{\rm cot}\Big(\frac{\pi\vartheta}{\beta_i}\Big)\int\limits_{0}^{\epsilon}e^{-p}{\rm log}(r_i)\,r_i\,dr_i=\\
=\frac{-\mu}{32\pi i\beta_i}\int\limits_{0}^{+\infty}dt\,e^{\mu t}\int\limits_{\tilde{\mathscr{C}}}d\vartheta\,\frac{{\rm cot}(\pi\vartheta/\beta_i)}{{\rm sin}^2(\vartheta/2)}\mathfrak{P}(\vartheta,t,\epsilon).
\end{split}
\end{align}
where $p,P$ are given by (\ref{new variable}) and
$$\mathfrak{P}(\vartheta,t,\epsilon):=\int\limits_{0}^{P}e^{-p}\big[{\rm log}p+{\rm log}t-{\rm log}\big({\rm sin}^2(\vartheta/2)\big)\big]\,dp.$$
Since $\Re{\rm sin}^2(\vartheta/2)\ge c_0>0$ on $\tilde{\mathscr{C}}$ and it grows exponentially as $\tilde{\mathscr{C}}\ni\vartheta\to\infty$, we have 
\begin{align*}
\mathfrak{P}(\vartheta,t,\epsilon)=\int\limits_{0}^{+\infty}e^{-p}\big[{\rm log}p+{\rm log}t-{\rm log}\big({\rm sin}^2(\vartheta/2)\big)\big]\,dp+\tilde{\mathfrak{P}}(\vartheta,t,\epsilon)=\\
={\rm log}t-{\rm log}\big({\rm sin}^2(\vartheta/2)\big)+\int\limits_{0}^{+\infty}e^{-p}{\rm log}p\,dp+\tilde{\mathfrak{P}}(\vartheta,t,\epsilon),
\end{align*}
where $\tilde{\mathfrak{P}}(\vartheta,t,\epsilon)$ and all its derivatives decay exponentially and uniformly in $\vartheta\in\tilde{\mathscr{C}}$ as $t\to +0$. Note that the integral in the last formula is equal to $\Gamma'(1)=-\gamma$. Thus, (\ref{integrated asymptotics log preliminary}) can be rewritten as
\begin{align}
\label{a mu log asymp}
\begin{split}
-2&\int\limits_{0}^{\epsilon}a_\mu(r_i)r_i{\rm log}(r_i) dr_i=\frac{\mu\mathfrak{Q}(\beta_i)}{\beta_i}\int\limits_{0}^{+\infty}dt\,e^{\mu t}{\rm log}t-\\
-&\frac{\mathfrak{Q}(\beta_i)}{\beta_i}\int\limits_{0}^{+\infty}e^{-p}{\rm log}p\,dp+\frac{\tilde{\mathfrak{Q}}(\beta_i)}{\beta_i}+O(|\mu|^{-\infty})=\\
=&\frac{\mathfrak{Q}(\beta_i)\,{\rm log}(-\mu)}{\beta_i}-\frac{\gamma}{6\beta_i}\Big(\frac{\beta_i}{2\pi}-\frac{2\pi}{\beta_i}\Big)+\tilde{\mathfrak{Q}}'(\beta_i)+O(|\mu|^{-\infty}),
\end{split}
\end{align}
where $\mathfrak{Q}$ is given by (\ref{Q integral}) and
$$\tilde{\mathfrak{Q}}(\beta):=\frac{1}{16\pi i\beta}\int\limits_{\tilde{\mathscr{C}}}\frac{{\rm cot}(\pi\vartheta/\beta)}{{\rm sin}^2(\vartheta/2)}{\rm log}\big({\rm sin}^2(\vartheta/2)\big)d\vartheta.$$
Here the integration contour can be replaced by the union of the lines $\pm \tilde{l}=\{\vartheta=\pm(\varepsilon-i\acute{\vartheta})\}_{\acute{\vartheta}\in\mathbb{R}}$ with arbitrary sufficiently small $\varepsilon>0$. Since the integrand is odd, one has
\begin{align}
\label{Q tilde integral}
\begin{split}
\tilde{\mathfrak{Q}}'(\beta)=\lim_{\varepsilon\to +0}&\int\limits_{-\infty}^{+\infty}\frac{i}{16\pi\beta}\frac{{\rm coth}(\pi(\acute{\vartheta}+i\varepsilon)/\beta)}{{\rm sinh}^2((\acute{\vartheta}+i\varepsilon)/2)}{\rm log}\big(-{\rm sinh}^2((\acute{\vartheta}+i\varepsilon)/2)\big)d\acute{\vartheta}=\\
=&\frac{1}{16}\mathcal{H}\int\limits_{0}^{+\infty}\frac{{\rm coth}(\pi\breve{\vartheta})\,d\breve{\vartheta}}{{\rm sinh}^2(\beta\breve{\vartheta}/2)}+\frac{1}{48\pi}-\frac{{\rm log}(\beta/2)}{12\beta}\Big(\frac{\beta}{2\pi}-\frac{2\pi}{\beta}\Big),
\end{split}
\end{align}
where $\mathcal{H}$ denotes the Hadamard regularization of the diverging integral,
$$\mathcal{H}\int\limits_{0}^{+\infty}\frac{{\rm coth}(\pi\breve{\vartheta})\,d\breve{\vartheta}}{{\rm sinh}^2(\beta\breve{\vartheta}/2)}=\lim_{\epsilon\to+0}\Bigg[\int\limits_{\epsilon}^{+\infty}\frac{{\rm coth}(\pi\breve{\vartheta})\,d\breve{\vartheta}}{{\rm sinh}^2(\beta\breve{\vartheta}/2)}-\frac{4}{\pi\beta^2\epsilon^2}-\frac{4{\rm log}\epsilon}{3\beta}\Big(\frac{\beta}{2\pi}-\frac{2\pi}{\beta}\Big)\Bigg].$$
Hence, 
\begin{align*}
\tilde{\mathfrak{Q}}(\beta)=-\frac{1}{8}\mathcal{H}\int\limits_{0}^{+\infty}&{\rm coth}(\pi\breve{\vartheta})\,{\rm coth}(\beta\breve{\vartheta}/2)\frac{d\breve{\vartheta}}{\breve{\vartheta}}-\\
-&\frac{{\rm log}(\beta/2)}{12}\Big(\frac{\beta}{2\pi}+\frac{2\pi}{\beta}\Big)+\frac{1}{12}\Big(\frac{3\beta}{4\pi}-\frac{2\pi}{\beta}\Big)
\end{align*}

The substitution of (\ref{a mu log asymp}) into (\ref{psi mu j beta asymp preliminary}) yields
\begin{align}
\label{psi mu j beta asymp}
\begin{split}
\Psi^{(i)}_\mu&=O(|\mu|^{-\infty})+\frac{\mathfrak{Q}(\beta_i)\,{\rm log}(-\mu)}{\beta_i}+\tilde{\mathfrak{Q}}'(\beta_i)-\\
-&\frac{1}{12}\Big(\frac{\beta_i}{2\pi}-\frac{2\pi}{\beta_i}\Big)\Big(\frac{{\rm log}|z_i-z_1|}{\pi}+\frac{2}{\beta_i}\Big[{\rm log}|B_i|+\gamma\Big]\Big).
\end{split}
\end{align}
The repeating of the above reasoning shows that the asymptotics for $\Psi^{(1)}_\mu$ as $\Re\mu\to-\infty$ is obtained by the replacement of $\xi_i,\beta_i,B_i$ with $\xi_1,\beta_1,B_1$ and changing the overall sign in (\ref{psi mu j beta asymp}). 

Substituting expressions (\ref{Psi j variation beta asymp}), (\ref{psi mu j beta asymp}) into the right-hand side of (\ref{det formula 1}) and taking into account (\ref{Taylir coefficients i}), the equality $\beta_j/2\pi-b_j=1$, one arrives at
\begin{align}
\label{IsthereBarnes}
\begin{split}
\frac{\partial {\rm log}\,({\rm det}\Delta/A)}{\partial\beta_i}&=\mathfrak{B}_i-\mathfrak{B}_1, \\ 
\mathfrak{B}_q:=\frac{1}{6}\sum_{j\ne q}\Big(\frac{1}{\beta_j}&+\frac{2\pi}{\beta_q^2}\Big)b_j{\rm log}|{z_j-z_q}|+\tilde{\mathfrak{Q}}'(\beta_q)+\frac{\pi\gamma}{3\beta_q^2}+\\
&+\frac{1}{6\beta_q}\Big(\frac{2\pi}{\beta_q}-\frac{\beta_q}{2\pi}\Big){\rm log}\Big|\frac{2\pi\sqrt{C}}{\beta_q}\Big|=\frac{\partial(\mathfrak{W}+\mathfrak{F}(\beta,C))}
{\partial\beta_q}
\end{split}
\end{align}
where $\tilde{\mathfrak{Q}}'$ and $\mathfrak{W}$ are defined in (\ref{Q tilde integral}) and (\ref{DoubleYouFunc}), respectively, and $\mathfrak{F}$ is given by
\begin{align}
\label{OnlyAngleFunc}
\begin{split}
\mathfrak{F}(\beta,C)=&\Bigg[\frac{1}{8}\mathcal{H}\int\limits_{0}^{+\infty}{\rm coth}(\pi\breve{\vartheta})\,{\rm coth}(\delta\breve{\vartheta}/2)\frac{d\breve{\vartheta}}{\breve{\vartheta}}+\\
&+\frac{1}{12}\Big(\frac{\delta}{2\pi}+\frac{2\pi}{\delta}\Big){\rm log}(2\pi^2 C/\delta)+\frac{1}{12}\Big(\frac{\delta}{4\pi}-\frac{4\pi}{\delta}\Big)+\frac{\pi\gamma}{3\delta}\Bigg]\Bigg|^{\delta=2\pi}_{\delta=\beta}.
\end{split}
\end{align}

\subsection{Derivation of $\partial {\rm log}\,({\rm det}\Delta/A)/\partial C$} 
For variation (\ref{scaling}), formulas (\ref{psi vertex int}), (\ref{modified psi mu asymp}) and (\ref{integrated asymptotic const term}), (\ref{Q integral})
\begin{align*}
\Psi^{(j)}_\mu:=-\frac{1}{C}\int_{\mathbb{K}_j(\epsilon)}\psi_\mu^{(0)}dS=\frac{1}{12 C}\Big(\frac{\beta_j}{2\pi}-\frac{2\pi}{\beta_j}\Big)+O(|\mu|^{-\infty})=\\
=\frac{\partial\mathfrak{F}(\beta_j,C)}{\partial C}+\frac{b_j}{6C}+O(|\mu|^{-\infty}).
\end{align*}
The substitution of the last expression into (\ref{det formula 1}) and taking into account that $\sum_{j=1}^{M}b_j=-2$ yields
\begin{equation}
\label{scaling det}
\partial_C{\rm log}\,({\rm det}\Delta/A)=\partial_C\Big[\sum_{j=1}^{M}\mathfrak{F}(\beta_j,C)-\frac{{\rm log}C}{3}\Big].
\end{equation}

\

\begin{rem}
Integration of formulas {\rm(\ref{Tankut Can})}, {\rm(\ref{IsthereBarnes})}, {\rm(\ref{scaling det})} leads to the Aurell-Salomonson formula 
\begin{equation}
\label{AS formula}
{\rm det}\Delta=AC^{-1/3}{\rm exp}\Big(\mathfrak{W}(z_1,\dots,z_M,\beta_1,\dots,\beta_M)+\sum_{j=1}^M\mathfrak{F}(\beta_j,C)-c\Big).
\end{equation}
The functions $\mathfrak{W}$ and $\mathfrak{F}$ are given by {\rm(\ref{DoubleYouFunc})} and {\rm(\ref{OnlyAngleFunc})}, respectively. The constant $c$ is ``global'', i.e., it is independent on all the parameters $z_j$, $\beta_j$, $C$ and on the number of vertices $M$ in {\rm(\ref{polyhedral metric})}. Thus, $c$ can be found by comparison of the expressions for ${\rm det}\Delta/A$ for the tetrahedron with all angles $\pi$ provided by {\rm (\ref{AS formula})} and {\rm (\ref{tetr})}, which yields
\begin{align*}
c={\rm log}(2^{2/3}\pi)+4\mathfrak{F}(\pi,1)
\end{align*}
due to {\rm(\ref{Tankut Can})}. Thus, we have arrived to
\begin{equation}
\label{AS formula 1}
{\rm det}\Delta=\frac{Area(X,m)}{(4C)^{1/3}\pi}{\rm exp}\Big(\mathfrak{W}(z_1,\dots,z_M,\beta_1,\dots,\beta_M)+\sum_{j=1}^M\mathfrak{F}(\beta_j,C)-4\mathfrak{F}(\pi,1)\Big).
\end{equation}
\end{rem}

\section{Justification of formula (\ref{det formula})}
\label{Justifisection}
\subsection*{Estimates of the eigenfunctions near vertices.} Denote by $\mathbb{K}$ the infinite cone with opening angle $\beta=\beta_j$ and by $(r,\varphi)$ the polar coordinates on it. Introduce the weighted spaces $H^l_{\upsilon}(\mathbb{K})$ ($l=0,1,\dots, \, \upsilon\in\mathbb{R}$) with the norms
\begin{equation}
\label{weighted norms}
\|v\|^2_{H^l_{\upsilon}(\mathbb{K})}=\sum_{p+q\le l}\int_{\mathbb{K}}\Big|r^{\upsilon-l}\partial_\varphi^{q}(r\partial_r)^{p}v\Big|^2\,rdr\,d\varphi
\end{equation}
and the model Laplacian $\vartriangle:=r^{-2}((r\partial_r)^2+\partial_\varphi^2)$.
\begin{prop}[see Chapter 2, \cite{NP}]
\label{Plamenprop}
The continuous operator 
$$\vartriangle:\,H^{l+2}_{\upsilon}(\mathbb{K})\to H^{l}_{\upsilon}(\mathbb{K})$$ 
is an isomorphism unless $\upsilon-l-1$ is multiple of $2\pi/\beta$.
\end{prop}
\begin{proof} Let $v\in H^{l+2}_{\upsilon}(\mathbb{K})\to H^{l}_{\upsilon}(\mathbb{K})$ and $\vartriangle v=r^{-2}f$. Introduce the new variable $\sigma={\rm log}r$ and the complex Fourier transform $\hat{v}(\tau,\varphi):=\frac{1}{\sqrt{2\pi}}\int_{-\infty}^{+\infty}e^{-i\sigma\tau}v(e^{\sigma},\varphi)d\sigma$, where $\Im\tau=\tau_0:=\upsilon-l-1$. Then the equations $(\partial_\varphi^2-\tau^2)\hat{v}(\tau,\cdot)=\hat{f}$ hold for almost all $\Re\tau$ and
\begin{equation}
\label{parceval}
\int\limits_{\Im\tau=\tau_0}\sum_{p+q\le l}\tau^{2p}\|\partial_\varphi^{q}\hat{v}(\tau,\cdot)\|^2_{L_2(\mathbb{R}/\beta\mathbb{Z})}d\tau\asymp\|v\|^2_{H^{l+2}_{\upsilon}(\mathbb{K})}
\end{equation}
due to the Parseval identity. It is easily checked (by a straightforward substitution) that 
$$\mathscr{R}_\tau(\varphi,\varphi')=-\frac{{\rm cosh}\big(\tau(|\varphi-\varphi'|-\beta/2)\big)}{2\tau {\rm sinh}(\beta\tau/2)}, \qquad \varphi,\varphi'\in \mathbb{R}/\beta\mathbb{Z}$$
is the kernel of the operator $(\partial_\varphi^2-\tau^2)^{-1}$ (acting in $L_2(\mathbb{R}/\beta\mathbb{Z})$). It is holomorphic in $\tau$ outside $\tau=\frac{2\pi i k}{\beta}$ ($k\in\mathbb{Z}$) and it obeys the estimate 
\begin{equation}
\label{RK estimatesat infinity}
|\tau|^2\|\mathscr{R}_\tau(\cdot,\cdot)\|_{L_2((\mathbb{R}/\beta\mathbb{Z})^2)}+|\tau|\|\partial_\varphi\mathfrak{R}_\tau(\cdot,\cdot)\|_{L_2((\mathbb{R}/\beta\mathbb{Z})^2)}\le c(l,\nu)
\end{equation}
for large $|\Re\tau|$. Since the function $\hat{v}(\tau,\varphi):=\int_{\Im\tau=\tau_0}\mathscr{R}_\tau(\varphi,\varphi')\hat{f}(\varphi')d\tau$ satisfies $\partial^{2k}_\varphi\hat{v}=\tau^{2k}\partial_\varphi\hat{v}+\sum_{p\le k}\tau^{2(k-p)}\partial^{2p}_\varphi f$, estimates (\ref{RK estimatesat infinity}) and (\ref{parceval}) imply the inequality $\|v\|_{H^{l+2}_{\upsilon}(\mathbb{K})}\le c\|\vartriangle u\|_{H^{l}_{\upsilon}(\mathbb{K})}$.
\end{proof}
Let $\chi\in C_c^{\infty}(\mathbb{K})$ and $\chi=1$ near the vertex. Then asymptotics (\ref{eigenfunc expansion near vert}) for the eigenfunction $u=u_k$ (corresponding to the eigenvalue $\lambda=\lambda_k$) near $z_j$ can be rewritten as $\chi\tilde{u}:=\tilde{v}\in H^{2M}_{\upsilon}(\mathbb{K})$, where
\begin{align*}
\tilde{u}:=u\circ\mathscr{Z}^{(j)}_t-u(z_j)J_{0}(r\sqrt{\lambda}), \quad \upsilon-2M+1\in\Big(-\frac{2\pi}{\beta},0\Big).
\end{align*}
In what follows, we assume that $\upsilon>0$. In view of Proposition \ref{Plamenprop}, definition (\ref{weighted norms}) of the weighted norms and the equation $(\vartriangle-\lambda)\tilde{u}=0$, we have
\begin{align*}
\|\tilde{v}\|_{H^{0}_{\upsilon-2M}(\mathbb{K})}\le \|\tilde{v}\|_{H^{2M}_{\upsilon}(\mathbb{K})}\le c\|\vartriangle\tilde{v}\|_{H^{2(M-1)}_{\upsilon}(\mathbb{K})}\le\dots\le c^M\|\vartriangle^M\tilde{v}\|_{H^{0}_{\upsilon}(\mathbb{K})}\le \\ 
\le c^M\Big(\lambda^M\|\tilde{v}\|_{H^{0}_{\upsilon}(\mathbb{K})}+\|[\vartriangle^M,\chi]\tilde{u}\|_{L_2(\mathbb{K}(\epsilon))}\Big)\le C\lambda^M\|\tilde{u}\|_{L_2(\mathbb{K}(\epsilon))}.
\end{align*}
Here and in the subsequent, all estimates are uniform in $\lambda=\lambda_k$ and $k$. Let $\mathscr{D}$ be a domain containing $\partial\mathbb{K}(\epsilon)$ and the support of $[\vartriangle^M,\chi]$ and the closure of $\mathscr{D}$ does not contains the vertex. Due to the smoothness increasing theorem for the Laplace operator we have $\|\tilde{u}\|_{H^{2M}(\mathscr{D})}\le C(\mathscr{D})\|\vartriangle^M\tilde{u}\|_{L_2(\mathscr{D})}$, whence
\begin{align}
\label{remest 1}
\|\tilde{v}\|_{H^{0}_{\upsilon-2M}(\mathbb{K})}\le C\lambda^M\|\tilde{u}\|_{L_2(\mathbb{K}(\epsilon))}\le C\lambda^M(\|u\|_{L_2(X;m)}+|u(z_j)|)
\end{align} 
due to the uniform boundedness of $J_{0}(r\sqrt{\lambda})$ for $r,\lambda>0$. Put $V_0=\beta^{-1}{\rm log}(r/\epsilon)$, then integration by parts yields
\begin{align*}
\lambda(u\circ\mathscr{Z}^{(j)}_t,V_0)_{L_2(\mathbb{K}(\epsilon)\backslash \mathbb{K}(\epsilon'))}=(\vartriangle u\circ\mathscr{Z}^{(j)}_t,V_0)_{L_2(\mathbb{K}(\epsilon)\backslash \mathbb{K}(\epsilon'))}=\\
=\frac{1}{\epsilon\beta}(u\circ\mathscr{Z}^{(j)}_t,1)_{L_2(\partial\mathbb{K}(\epsilon)}-u(z_j)+o(1) \qquad (\epsilon'\to 0).
\end{align*}
Since $V_0$ is square integrable on $\mathbb{K}(\epsilon)$, the last formula, the Sobolev trace theorem, and the estimate $\|u\|_{H^1(\mathscr{D})}^2\asymp(\Delta u,u)_{L_2(X;m)}$ yield
\begin{equation}
\label{remest 2}
|u(z_j)|\le c(\|u\|_{L_2(X;m)}+\sqrt{(\Delta u,u)_{L_2(X;m)}})=O(\sqrt{\lambda}) \qquad (\lambda\to\infty).
\end{equation}
Since $\|u\|_{L_2(X;m)}=1$, estimates (\ref{remest 1}), (\ref{remest 2}) imply
\begin{equation}
\label{remest}
\|\tilde{v}\|_{H^{0}_{\upsilon-2M}(\mathbb{K})}=O(\lambda^{M+\frac{1}{2}}) \qquad (\lambda\to\infty).
\end{equation}
\subsection*{Estimates of regularized integrals for $\dot{\lambda}_k/\lambda_k$.} Recall that variation (\ref{vert shift}) is of the form $\dot{\phi}=c\xi_i^{-\frac{2\pi}{\beta_i}}+O(1)$ due to (\ref{conical coordinates}). Since the integration over $\varphi={\rm arg}\xi_i$ kills the terms $\xi_i^{-\frac{2\pi}{\beta_i}}J_{0}(r\sqrt{\lambda})^p$ ($p=1,2$), we have
\begin{equation}
\begin{split}
{\rm p.v.}\int\limits_{r_i\le\epsilon}\dot{\phi}u_k^2dS={\rm p.v.}\int\limits_{r_i\le\epsilon}\dot{\phi}[u_k^2-(u(z_j)J_{0}(r\sqrt{\lambda}))^2]dS=\\
=\int\limits_{r_i\le\epsilon}(\tilde{v}+2u(z_j)J_{0}(r\sqrt{\lambda}))\tilde{v}\,\big(\xi_i^{-\frac{2\pi}{\beta_i}}+O(1)\big)r_idr_id\varphi_i
\end{split}
\end{equation}
Then above estimates (\ref{remest}), (\ref{remest 2}) imply
$$\left|{\rm p.v.}\int\limits_{r_i\le\epsilon}\dot{\phi}u_k^2dS\right|\le c\|y\|_{H^0_{\upsilon-2M}(\mathbb{K})}(1+|u(z_j)|)+O(1)=O(\lambda^{M+1}) \qquad (\lambda\to\infty)$$
At the same time,
$$\left|\int\limits_{r_i\ge\epsilon}\dot{\phi}u_k^2dS\right|\le\max_{r_i\ge\epsilon}|\dot{\phi}|\,\|u_k\|^2_{L_2(X;m)}=O(1) \qquad (\lambda\to\infty).$$
Combining the last two inequalities, one arrives at
\begin{equation}
\label{var ind eig est}
\left|{\rm p.v.}\int\limits_{X}\dot{\phi}u_k^2dS\right|=O(\lambda^{M+1}) \qquad (\lambda\to\infty, \quad M=[\pi/\beta]+1).
\end{equation}
If the variation $\dot{\phi}$ is of form (\ref{vert opening}), then the same estimate is obtained in an even simpler way (and for smaller $M$) due to the weaker (logarithmic) singularity of $\dot{\phi}$. Note that all the above estimates are uniform in the parameter $t$.

\subsection*{Differentiability of $\zeta_{\Delta-\mu}(2+q)$ in $t$ for large positive $q$.}
Formula (\ref{variation of eigenvalues conic}) implies that
\begin{equation}
\label{zeta of big}
\begin{split}
\partial_t((\lambda_k-\mu)^{-(2+q)})=&\frac{-(2+q)\lambda_k}{(\lambda_k-\mu)^{3+q}}\,\mathscr{J}\,{\rm p.v.}\int\limits_X\dot{\phi}u_k^2dS=\\
&=\frac{-1}{(1+q)!}\,\mathscr{J}\,{\rm p.v.}\int\limits_X \partial_\mu^{2+q}\Big(\mu \frac{u_k^2}{\lambda_k-\mu}\Big)\dot{\phi}dS.
\end{split}
\end{equation}
In view of (\ref{zeta of big}) and (\ref{var ind eig est}) and the Weyl's law $\lambda_k\sim k$, we arrive to the (uniform in $t\in(-t_0,t_0)$ and $k$) estimate
$$|\partial_t((\lambda_k-\mu)^{-(2+q)})|=O(k^{[1/2(b+1)]-q}) \qquad (\lambda\to\infty),$$
where $b_{\rm min}=\min_j b_j$. Due to the last formula, for $q>[1/2(b+1)]+1$, the series $\sum_k\partial_t((\lambda_k-\mu)^{-(2+q)})$ converge uniformly in $t\in(-t_0,t_0)$ and thus the series $\zeta_{\Delta-\mu}(2+q):=\sum_k((\lambda_k-\mu)^{-(2+q)})$ admit term-wise differentiation in $t$. In particular, making summation over $k$ in (\ref{zeta of big}) yields
\begin{equation}
\label{zeta of big variation}
\dot{\zeta}_{\Delta-\mu}(2+q)=\frac{1}{(1+q)!}\,\mathscr{J}\,{\rm p.v.}\int\limits_X \dot{\phi}\partial_\mu^{2+q}\psi_\mu dS=\frac{\partial_\mu^{2+q}\Psi_\mu}{(1+q)!},
\end{equation}
for $q>[1/2(b+1)]+1$; here $\psi_\mu$ and $\Psi_\mu$ are defined in (\ref{psi mu}), (\ref{psi int}), respectively. Note that the operator $\mathscr{J}$ eliminating removable discontinuities can be omitted in (\ref{zeta of big variation}) since $\partial_\mu^{2+q}\Psi_\mu$ is continuous in $t$. Indeed, due to (\ref{psi int 1}), (\ref{psi vertex int}), it is sufficient to show that ${\rm p.v.}\int_{\mathbb{K}_j(\epsilon)}\dot{\phi}\partial_\mu^{2+q}\psi^{(0)}_{\mu}dS$ is continuous in $t$. To this end, one applies expansion (\ref{modified psi mu asymp}) for $\psi^{(0)}_{\mu}$, where first term (\ref{modified psi mu main term}) is rotationally symmetric (and, thus, should be killed by the integration with the singular part of $\dot{\phi}$) while the remainder is of the form $\tilde{\psi}^{(0)}_{\mu}(z)=\tilde{R}_{\mu}(z,z)$, where $\tilde{R}_{\mu}(\cdot,z)$ obeys asymptotics (\ref{eigenfunc expansion near vert}) near $z_j$. Thus, $\int_{\mathbb{K}_j(\epsilon)}\dot{\phi}\partial_\mu^{2+q}[\tilde{\psi}^{(0)}_{\mu}(z)-\tilde{\psi}^{(0)}_{\mu}(z_j)]dS$ is continuous in $t$.

Note that formulas (\ref{psi int 1}), (\ref{psi vertex int}), (\ref{Psi j variation z asymp fin}), (\ref{Psi i variation z asymp fin}), (\ref{psi mu j beta asymp}) lead to the (admitting differentiation in $\mu$) asymptotics 
\begin{equation}
\label{asymp of Psi integrated}
\Psi_\mu=C_{1}\mu{\rm log}(-\mu)+c_1\mu+C_0{\rm log}(-\mu)+c_0+O(\mu^{-\infty}) \qquad (\Re\mu\to-\infty),
\end{equation}
where the coefficients depend on the parameters of metric (\ref{polyhedral metric}) and its variations; in particular,
\begin{equation}
\label{asymp of Psi integrated 1}
C_1=-\dot{A}/4\pi, \qquad C_0=-\partial_t\big(\sum_{k}\mathfrak{Q}(\beta_k)\big)
\end{equation}
(here the equality $\partial_t(\sum_k\beta_k)=0$ is used).

\subsection*{Differentiability of $\zeta_{\Delta}(s)$ in the parameter for large positive $\Re s$.} In view of the residue theorem, we have
$$(s-1)\dots(s-1-q)\lambda_k^{-s}=\frac{(q+1)!}{2\pi i}\int_\Gamma\frac{\mu^{1+q-s}d\mu}{(\lambda_k-\mu)^{2+q}},$$
where $\Gamma$ is the contour enclosing the cut $(-\infty,0]$. Making summation over $k$ and taking into account the Weyl's law, one arrives at
\begin{equation}
\label{zeta of s via zeta of big}
(s-1)\dots(s-1-q)\zeta_{\Delta}(s)=\frac{(q+1)!}{2\pi i}\int\limits_\Gamma\zeta_{\Delta-\mu}(2+q)\mu^{1+q-s}d\mu
\end{equation}
for $\Re s>1$ (for $\Re s\le 1$, both sides of this formula should be understood as analytic continuations of them from the half-plane $\Re s>1$). For sufficiently large positive $q$ and $\Re s-q$, the right-hand side of (\ref{zeta of s via zeta of big}) admits the differentiation in $t$ due to (\ref{zeta of big variation}). Thus, for such $q,s$, one can write 
\begin{equation}
\label{zeta of s var via zeta of big}
(s-1)\dots(s-1-q)\dot{\zeta}_{\Delta}(s)=\frac{1}{2\pi i}\int\limits_\Gamma\mu^{1+q-s}\partial_\mu^{2+q}\Psi_\mu d\mu
\end{equation}
Integrating by parts in (\ref{zeta of s var via zeta of big}) and taking into account that $\partial_\mu^{2+l}\Psi_\mu=O(|\mu|^{-l})$ due to (\ref{asymp of Psi integrated}), one finally arrives at
\begin{equation}
\label{zeta of s var via zeta of big to 2}
(s-1)\dot{\zeta}_{\Delta}(s)=\frac{1}{2\pi i}\int\limits_\Gamma\mu^{1-s}\partial_\mu^{2}\Psi_\mu d\mu.
\end{equation}

\subsection*{Global differentiability of $\zeta_{\Delta}(s)$ in $t$ and the justification of (\ref{det formula}).} Now we should to prove that the same formulas are valid for all $s\in\mathbb{C}$, i.e. that one can interchange the analytic continuation of $\zeta_{\Delta}(s)$ and its differentiation in $t$. 

To this end, recall that the zeta function of $\Delta-\mu$ is related to its heat trace $K(\tau|\Delta):=\sum_k e^{-\lambda_k \tau}=\int_X H_\tau(x,x)dS$ (where $H_\tau$ is the heat kernel of $\Delta$) via
\begin{equation}
\label{zetazeta}
\zeta_{\Delta-\mu}(s)=\frac{1}{\Gamma(s)}\int\limits_{0}^{+\infty}e^{\mu\tau}\tau^{s-1}K(\tau|\Delta)d\tau.
\end{equation}
Using model heat kernel (\ref{HK Cone}) in the cone as a parametrix for $H_\tau$, one deduces the asymptotics
\begin{equation}
\label{heat kernal asyympp}
K(\tau|\Delta_{D})=\frac{A}{4\pi \tau}-\sum_{k}\mathfrak{Q}(\beta_k)+O(e^{-\mathfrak{b}/\tau}), \qquad \tau\to +0,
\end{equation}
where $\mathfrak{b}>0$ (see Theorem 7, \cite{KokKor}). Then there holds the aymptotics
\begin{equation}
\label{excluding singular part 0}
\zeta_{\Delta-\mu}(s)-\check{\zeta}_\mu(s)=\frac{1}{\Gamma(s)}\int\limits_{0}^{+\infty}\tau^{s-1}e^{\mu\tau}\Big(K-\frac{A}{4\pi \tau}+\sum_{k}\mathfrak{Q}(\beta_k)\Big)d\tau=O(\mu^{-\infty})
\end{equation}
as $\mu\to-\infty$, where the right-hand side is well-defined for all $s\in\mathbb{C}$ while
\begin{equation}
\label{excluding singular part 1}
\check{\zeta}_\mu(s):=\frac{A\Gamma(s-1)}{4\pi\Gamma(s)(-\mu)^{s-1}}-\frac{\sum_{k}\mathfrak{Q}(\beta_k)}{(-\mu)^{s}}.
\end{equation}
Comparison of asymptotics (\ref{zeta of big variation}), (\ref{asymp of Psi integrated}), (\ref{asymp of Psi integrated 1}) with (\ref{excluding singular part 0}), (\ref{excluding singular part 1}) yields
\begin{equation}
\label{comparison asymptotics just}
\partial_t[\zeta_{\Delta-\mu}(2+q)-\check{\zeta}_\mu(2+q)]=\Big[\frac{\partial_\mu^{2+q}\Psi_\mu}{(1+q)!}-\partial_t\check{\zeta}_\mu(2+q)\Big]=O(\mu^{-\infty})
\end{equation}
as $\mu\to-\infty$. Thus, the integrals 
\begin{align*}
\frac{(q+1)!}{2\pi i}\int\limits_{\Gamma}\mu^{1+q-s}&\Big[\zeta_{\Delta_D-\mu}(2+q)-\check{\zeta}_\mu(2+q)\Big]d\mu=\\
&=(s-1)\dots(s-1-q)\Big[\zeta_{\Delta-\mu}(s)-\check{\zeta}_\mu(s)\Big],\\
\frac{(q+1)!}{2\pi i}\int\limits_{\Gamma}\mu^{1+q-s}&\frac{d}{dt}\Big[\zeta_{\Delta_D-\mu}(2+q)-\check{\zeta}_\mu(2+q)\Big]d\mu
\end{align*}
converge for all $s\in\mathbb{C}$ uniformly in the parameter $t$. Thus, the usual Leibniz integration rule yields
\begin{align*}
(s-1)\dots&(s-1-q)\frac{d}{dt}\Big[\zeta_{\Delta-\mu}(s)-\check{\zeta}_\mu(s)\Big]=\\
=&\frac{(q+1)!}{2\pi i}\int\limits_{\Gamma}\mu^{1+q-s}\frac{d}{dt}\Big[\zeta_{\Delta_D-\mu}(2+q)-\check{\zeta}_\mu(2+q)\Big]d\mu,
\end{align*}
whence formulas (\ref{zeta of s var via zeta of big}) and (\ref{zeta of s var via zeta of big to 2}) follow. Finally, the application of Lemma \ref{magiclemma} yields formula (\ref{det formula}). (Note that the above regularization trick justifies also the change of order of the analytic continuation and the differentiation in the parameter $t$ in formula (\ref{zeta of s via zeta of 2}) in the case of smooth metrics.) 

Thus, formula (\ref{det formula}) and all calculations made after (\ref{det formula}) are justified.

\appendix

\section{Proof of Lemma \ref{magiclemma}}
For sufficiently large $\Re s$, the function $\widehat{\Psi}(s)$ is well defined and
\begin{equation}
\label{general regularization 1}
\widehat{\Psi}(s)=\pi^{-1}e^{-i\pi s}{\rm sin}(\pi s)J_\infty(s)+J_0(s),
\end{equation}
where
\begin{equation}
\label{J 0 j infty}
J_0(s):=\frac{1}{2\pi i}\int\limits_{|\mu|=\epsilon}\mu^{1-s}\partial_\mu^2 \Psi(\mu)d\mu, \qquad J_\infty(s):=\int\limits_{-\infty-i0}^{-\epsilon-i0}\mu^{1-s}\partial_\mu^2 \Psi(\mu) d\mu
\end{equation}
(here $0<\epsilon<\lambda_1$). Since the contour $|\mu|=\epsilon$ is compact, $J_0$ is holomorphic on $\mathbb{C}$ and $J_0(0)=0$. Integration by parts in (\ref{J 0 j infty}) yields
\begin{align}
\label{J 0 in}
-\partial_s J_0(s)&=\mu^{-s}\big[\mu\partial_\mu+s-1\big]\Psi(\mu)\Big|^{\mu=-\epsilon}+O(s)+\Psi(0),\\ \nonumber
J_\infty(s)&=\mu^{-s}[\mu\partial_\mu+s-1]\Psi(\mu)\Big|^{\mu=-\epsilon}+s(s-1)\int\limits_{-\infty-i0}^{-\epsilon-i0}\frac{\Psi(\mu)}{\mu^{s+1}}d\mu.
\end{align}
Now, the integration in (\ref{F asymp}) yields
\begin{align}
\label{meromorphic ext}
\begin{split}
\int\limits_{-\infty-i0}^{-\epsilon-i0}\frac{\Psi(\mu)}{\mu^{s+1}}d\mu-\int\limits_{-\infty-i0}^{-\epsilon-i0}\frac{\Phi(\mu)}{\mu^{s+1}}d\mu&=\\
=\sum_{k=1}^K (-\epsilon-i0)^{r_k-s}\Big[\frac{\Psi_k}{r_k-s}+&\frac{\epsilon[1+(s-r_k-1){\rm log}\epsilon]\tilde{\Psi}_k}{(s-r_k-1)^2}\Big],
\end{split}
\end{align}
where the second integral in the left-hand side is well-defined and holomorphic for $\Re s>\kappa$ due to the estimate $|\mu^k\partial^k_\mu\Phi(\mu)|=O(\mu^{\kappa})$. Then the first integral in (\ref{meromorphic ext}) can be extended meromorphically to the half-plane $\Re s>\kappa$. Since the right-hand side of (\ref{meromorphic ext}) is
$$-\frac{\tilde{\Psi}(\infty)}{s^2}-\frac{\Psi(\infty)+\pi i\tilde{\Psi}(\infty)}{s}+O(1)$$
near $s=0$, $J_\infty$ can be extended meromorphically to the neighborhood of $s=0$, and 
\begin{equation}
\label{J inf}
J_\infty(s)=\mu^{-s}[\mu\partial_\mu+s-1]\Psi(\mu)\Big|^{\mu=-\epsilon}+\Psi(\infty)+(s^{-1}-1+\pi i)\tilde{\Psi}(\infty)+O(s).
\end{equation}
The substitution of (\ref{J inf}) into (\ref{general regularization 1}) yields $\widehat{\Psi}(0)=\tilde{\Psi}(\infty)$. Similarly, the differentiation of (\ref{general regularization 1}) and taking into account formulas (\ref{J 0 in}), (\ref{J inf}) and the estimate $x^{-1}{\rm sin}(x)-1=O(x^2)$ leads to
\begin{align*}
[\partial_s\widehat{\Psi}](0)=[&J_\infty(s)-s^{-1}\tilde{\Psi}(\infty)]_{s=0}+\partial_s J_0(0)-\\
-&i\pi\tilde{\Psi}(\infty)=\Psi(\infty)-\tilde{\Psi}(\infty)-\Psi(0).
\end{align*}
Lemma \ref{magiclemma} is proved.
\qed

\section{Proof of Lemmas \ref{SPD lemma} and \ref{SPD lemma mult}}
\label{SPD theory}
\subsection*{Solutions and the DN maps for model cones.} Let $\mathbb{K}$ be the (infinite) cone of opening $\beta$ and let $\mathbb{K}(\epsilon)\subset\mathbb{K}$ is the $\epsilon$-neighborhood of its vertex. Let $\Delta^{\mathbb{K}}$ be the Laplacian on $\mathbb{K}$ and let $\Delta^{\mathbb{K}(\epsilon)}$ be the (Friedrichs) Dirichlet Laplacian on $\mathbb{K}(\epsilon)$. Then expression (\ref{eigenfunc expansion near vert}) is the exact (bounded) solution to the equation $(\Delta^{\mathbb{K}}-\lambda)u=0$ in $\mathbb{K}(\epsilon)$ with the Dirichlet data $u|_{\partial\mathbb{K}(\epsilon)}=f$. Indeed, each function $u(r,\varphi)=J_{2\pi|k|/\beta}(r\sqrt{\lambda})e^{2\pi k i\varphi/\beta}$ (where $(r,\varphi)$ are the polar coordinates on $\mathbb{K}$) obeys $(\Delta^{\mathbb{K}}-\lambda)u=0$ in $\mathbb{K}$ while formula (\ref{eigenfunc expansion near vert}) becomes the Fourier expansion of the boundary data $u|_{\partial\mathbb{K}(\epsilon)}=f$ if $r=\epsilon$. 

Note that $\lambda$ is the eigenvalue $\Delta^{\mathbb{K}(\epsilon)}$ if and only if one of the denominators $J_{2\pi|k|/\beta}(\epsilon\sqrt{\lambda})$ in (\ref{eigenfunc expansion near vert}) equals zero. Slightly changing $\epsilon$, one can assume that (fixed) $\lambda$ is not an eigenvalue of $\Delta^{\mathbb{K}(\epsilon)}$.

Recall that the Bessel functions admit the expansions
$$J_\nu(z)=z^\nu\tilde{J}_\nu(z)=z^\nu\sum_{j=0}^{\infty}\mathfrak{c}_k(\nu)z^{2j},$$
where the coefficients
$$\mathfrak{c}_j(\nu):=\frac{(-1)^j}{2^{\nu+2j}j!\Gamma(j+\nu+1)}$$
and their derivatives with respect to $\nu$ decay super-exponentially as $j,\nu\to+\infty$ due to the Stirling expansion 
$${\rm log}\Gamma(z)=(z-1/2){\rm log}z-z+1/2\pi+\dots, \qquad (\Re z>0).$$
Thus, the coefficients $c_{j}(\nu)$ $(j>0)$ in {\rm(\ref{eigenfunc expansion near vert})} and their derivatives with respect to $\nu$ decay super-exponentially as $j\to+\infty$ or $\nu\to+\infty$.

The Dirichlet-to-Neumann map $\Lambda=\Lambda(\lambda,\epsilon,\beta): \ u|_{\partial\mathbb{K}(\epsilon)}=(\partial_r u)|_{\partial\mathbb{K}(\epsilon)}$ associated with the equation problem $(\Delta^{\mathbb{K}}-\lambda)u=0$ admits the expression
\begin{equation}
\label{DN cone}
\Lambda f=\frac{\sqrt{\lambda}}{\beta\epsilon}\sum_{k\in\mathbb{Z}}\frac{J'_{2\pi|k|/\beta}(\epsilon\sqrt{\lambda})}{J_{2\pi|k|/\beta}(\epsilon\sqrt{\lambda})}e^{2\pi k i\varphi/\beta}(f,e^{2\pi k i\varphi/\beta})_{L_2(\partial \mathbb{K}(\epsilon))}
\end{equation}
obtained by the differentiation of (\ref{eigenfunc expansion near vert}) in $r$. Due to the aforementioned, straightforward but cumbersome calculations show that $\Lambda$ depends real-analytically on $\beta$, i.e., it admits the expansion in the (converging in $B(H^{1/2}(\partial\mathbb{K}(\epsilon);\mathbb{R});H^{-1/2}(\partial\mathbb{K}(\epsilon);\mathbb{R})$) Taylor series 
$$\Lambda(\lambda,\epsilon,\beta)=\Lambda(\lambda_0,\epsilon,\beta_0)+\sum_{l+s>0}\Lambda^{(l,s)}(\lambda_0,\epsilon,\beta_0)(\beta-\beta_0)^l (\lambda-\lambda_0)^s,$$ 
whose coefficients $\Lambda^{(k,s)}$ are order one pseudo-differential operators. 

\subsection*{Parametrix for the resolvent kernel.} Suppose that $\lambda(0)$ is not an eigenvalue of $\Delta_t$ and $y\in X$ is not a vertex. Then there are holomorphic coordinates $z=z_t$ near $y$ in which the metric $m$ is of the form $m=|dy|^2$. Introduce the singular part 
$$R^{sing}_{t,\lambda}(x,y):=\frac{1}{2\pi}\chi(d)K_0(d\sqrt{-\lambda}) \qquad (d=d_t(x,y))$$  
of the resolvent kernel $R_\lambda(x,y)=R_{t,\lambda}(x,y)$ of $\Delta=\Delta_t$. Here we assumed that the cut for the logarithm ${\rm log}z$ avoids the small neighborhood containing $z=-\lambda$. In addition, $\chi$ is a smooth cut-off function equal to one near the origin and its support is small.
Put 
$$p=p_{t,\lambda,y}(x):=-[\Delta_{t,x},\chi(d)]\frac{1}{2\pi}K_0(d\sqrt{-\lambda}).$$ 
Then the resolvent kernel can be represented as $R_{t,\lambda}=R^{sing}_{t,\lambda}+u_{t,\lambda,y}$, where the remainder $u=u_{t,\lambda,y}$ obeys the equation $(\Delta-\lambda)u=p$.

\subsection*{Reduction to a boundary value problem in a regular fixed domain.} Introduce the notation $u_j:=u|_{\mathbb{K}_j(\epsilon)}$, ($j=1,\dots,M$) and $u_0:=u|_{X(\epsilon)}$, where $X(\epsilon)$ the complement of all $\mathbb{K}_j(\epsilon)$ in $X$.  As mentioned above, one can chose (small) $\epsilon$ in such a way that $\lambda$ is separated from the spectra of all Dirichlet Laplacians in $K_j(\epsilon)$, $j=1,\dots,M$ while ${\rm supp}p$ belongs to $X(\epsilon)$. Then formula (\ref{eigenfunc expansion near vert}) with $u=u_j$ describes all the restrictions $u_j$ ($j>0$) while $u_0$ is the solution to the problem
\begin{equation}
\label{reduced problem}
(\Delta-\lambda)u_0=p \text{ in } X(\epsilon), \quad (\Lambda_j+\partial_\nu)u_0=0 \text{ on } \partial\mathbb{K}_j(\epsilon) \quad (j=1,\dots,M).
\end{equation}
Here $\nu$ is the exterior unit normal vector on $\partial X(\epsilon)$ while $\Lambda_j$ is the Dirichlet-to-Neumann map for the cone $\mathbb{K}(\epsilon)=\mathbb{K}_j(\epsilon)$ given by (\ref{DN cone}) with $(r,\varphi)=(r_j,\varphi_j)$, $\beta=\beta_j$. Conversely, if $u_0$ is a solution to (\ref{reduced problem}), then it admits a continuation to the solution $u$ to $(\Delta-\lambda)u=p$ on $X$ which is given by expression (\ref{eigenfunc expansion near vert}) in each $K_j(\epsilon)$, where $(r,\varphi)=(r_j,\varphi_j)$, $\beta=\beta_j$ and $f_j(\varphi_j):=u_0|_{\partial\mathbb{K}_j(\epsilon)}$.

Now, we introduce the small parameter $t$ and allow the coefficients $z_j,\beta_j$ in metric (\ref{polyhedral metric}) as well as the parameters $y,\lambda$ to depend real-analytically on $t$. We assume that $\lambda(0)$ is not an eigenvalue of $\Delta_0$ (the $\lambda(t)$ is not an eigenvalue of $\Delta_t$ for small $t$). Although the domain $X(\epsilon)=X(\epsilon,t)$ in (\ref{reduced problem}) depend on $t$, one can find a real-analytic family of smooth diffeomorphisms
$$\varsigma=\varsigma_t: \ X(\epsilon,t)\mapsto X(\epsilon,0)$$ 
to transfer problem (\ref{reduced problem}) to a fixed domain $X(\epsilon,0)$. 

(An example of such family can be constructed as follows. Let $\chi_j$ be a smooth non-negative cut-off function on $X$ equal to one in $\mathbb{K}_j(2\epsilon)$ and supported in $\mathbb{K}_j(3\epsilon)$. We define $\varsigma_t$ in such a way that $\varsigma_t$ is the identity on $X(3\epsilon,0)$ and 
$$\varsigma_t(z):=\chi_j(z)\cdot\mathscr{Z}^{(j)}_0\circ\mathfrak{T}_{\beta_j(0),\beta_j(t)}\circ{\mathscr{Z}^{(j)}_t}^{-1}+(1-\chi_j(z))\cdot z$$
for $z\in\mathbb{K}_j(3\epsilon)$; here $\mathscr{Z}^{(j)}_t$ is defined by (\ref{loc coord map}) and $\mathfrak{T}_{\beta,\beta'}$ is defined after (\ref{loc coord map}). Then $\varsigma_t$ is a diffeomorphism for small $t$ and it is real-analytic in $t$.)

\subsection*{Perturbation series.} In the subsequent, we keep the same notation for the Laplacian and the DN-maps transferred (by means of $\varsigma_t$) to $X(\epsilon,0)$, i.e., we consider problem (\ref{reduced problem}) in the fixed domain while the (induced) metric $g(t)$ on $X(\epsilon)=X(\epsilon,0)$, the normal vector $\nu$ and the pseudo-differential operators $\Lambda_j(t)$ in (\ref{reduced problem}), and the function $p$ depend real-analytically on $t$,
$$\Delta_t+\lambda_t=\sum_{k=0}^\infty A_k t^k, \quad \Lambda_j(t)+\partial_\nu(t)=\sum_{k=0}^\infty L_k t^k \text{ on each } \partial\mathbb{K}_j(\epsilon), \quad p=\sum_{k=0}^\infty p_k t^k.$$
We seek a formal solution to (\ref{reduced problem}) of the form $u_0=\sum_{k=0}^\infty v^{(k)}t^k$. Substituting the above expansions into (\ref{reduced problem}), one obtains the sequence of problems
\begin{equation}
\label{reduced problem step}
A_0 v_l=p_l-\sum_{k=1}^l A^k v_{l-k} \text{ in } X(\epsilon), \qquad L_0 v_l=-\sum_{k=1}^l L^k v_{l-k} \text{ on } \partial X(\epsilon).
\end{equation}
For $f\in C^{\infty}(\partial X(\epsilon);\mathbb{R})$, introduce the function 
$$Y_t[f](x):=f(s)\varrho\chi(\varrho),$$ 
where $(s=s_t,\varrho=\varrho_t)$ are semi-geodesic coordinates near $\partial X(\epsilon)$, i.e. $\varrho_t(x)$ is the distance from $x\in X(\epsilon)$ to $\partial X(\epsilon)$ while $s_t(x)$ is the closest to $x$ point of $\partial X(\epsilon)$. By replacing $v_l$ with $v_l-Y_0[\sum_{k=1}^l L^k v_{l-k}]$, one can reduce (\ref{reduced problem step}) to the problem of the form $A_0 y=\tilde{p}$ in $X(\epsilon)$, $L_0 y=0$ on $\partial X(\epsilon)$, which is solvable due to our assumption that $\lambda(0)$ does not belong to the spectrum of $\Delta_0$. Thus, each problem (\ref{reduced problem step}) is uniquely solvable.

\subsection*{Diefferentiability of $\mathcal{F}_t(\cdot,\cdot|\mathcal{E})$.} Now, put $$U_M:=\sum_{l=0}^M v_l-Y_t\Big[(\Lambda_j(t)+\partial_\nu(t))\Big(\sum_{l=0}^M v_l\Big)\Big],$$ 
then $(\Delta_t+\lambda_t)(U_M-u_0)$ and all its derivatives in $x$ decay as $O(t^{M+1})$ uniformly in $X(\epsilon)$ as $t\to 0$ while $(\Lambda_j(t)+\partial_\nu(t))(U_M-u_0)=0$ on $X(\epsilon)$. In particular, $U_M-u_0$ admits the extension (by formulas (\ref{eigenfunc expansion near vert})) on $X$ obeying $(\Delta_t-\lambda_t)(U_M-u_0)=0$ in each $\mathbb{K}_j(\epsilon)$. For extended $U_M-u_0$ one has $\|(\Delta_t-\lambda_t)(U_M-u_0)\|_{L_2(X;m_t)}=O(t^{M+1})$. Since $\lambda_0$ is not an eigenvalue of $\Delta_0$, the distance between $\lambda_t$ and the spectrum of $\Delta_t$ is positive for all small $t$ and the last estimate implies $\|U_M-u_0\|_{L_2(X;m_t)}=O(t^{M+1})$. The above inequalities and local estimates of solutions to elliptic equations (provided by the increasing smoothness theorems) imply that $U_M-u_0$ all its derivatives in $x$ decay as $O(t^{M+1})$ uniformly in $X(2\epsilon)$ as $t\to 0$. Hence $u_0$ is differentiable in $t,x$ on $X(\epsilon)$ and, since $\epsilon>0$ is arbitrary, the resolvent kernel $R^{(t)}_\lambda(x,y)$ is differentiable in $(x,y,\lambda,t)$ outside the diagonal $(x,x,\lambda,t)$, the vertices $(z_k(t),z_j(t),\lambda,t)$ and the poles $(x,y,\lambda_j(t),t)$. Repeating the same perturbation method for the terms $v_1=\partial_t u_0$ e.t.c., one proves the infinite differentiability of $u_{t,\lambda,y}=R_{t,\lambda}-R^{sing}_{t,\lambda}$ in the parameters.

Suppose that the domain $U\subset\mathbb{C}$ contains eigenvalues $\lambda_j(0),\dots,\lambda_{j+m-1}(0)$ while its boundary $\partial U$ does no intersect with the spectrum of $\Delta_0$; then $U$ contain $\lambda_j(t),\dots,\lambda_{j+m}(t)$ for sufficiently small $t$. Let the function $\mathcal{E}$ be holomorphic in the neighborhood of $\overline{U}$. Due to the residue theorem, function (\ref{F residue func}) obeys
\begin{align*}
\mathcal{F}_t(x,y|\mathcal{E})&=\sum_{j=0}^{m-1}\mathcal{E}(\lambda_{k+j}(t))u_{k+j}(x)u_{k+j}(y)=\\
=&\int_{\partial U}\Big[\mathcal{E}(\lambda)R_{t,\lambda}(x,y)\Big]d\lambda=\int_{\partial U}\Big[\mathcal{E}(\lambda)u_{t,\lambda,y}(x)\Big]d\lambda.
\end{align*}
(here we used the fact that $K_0(d\sqrt{-\lambda})$ is holomorphic in the neighborhood of the ray $[\epsilon,+\infty)$ for any $\epsilon>0$). In particular, $\mathcal{F}_t(x,y|\mathcal{E})$ is smooth in $(x,y,t)$ outside vertices $(z_k(t),z_j(t),t)$. Thus, we have proved statement {\it (1)} of Lemma \ref{SPD lemma mult}. 

Suppose, for a while, that $U$ contains only one simple eigenvalue $\lambda_j$. Then 
$$\lambda_j=\frac{\mathcal{F}_t(y_0,y_0|\lambda)}{\mathcal{F}_t(y_0,y_0|1)}, \qquad u_j(x)=\frac{\mathcal{F}_t(x,y_0|1)}{\sqrt{\mathcal{F}_t(y_0,y_0|1)}}$$
are smooth in $t$ provided that $\mathcal{F}_0(y_0,y_0|1)>0$ and the eigenfunction $u_j$ is chosen in such a way that $u_j(y_0)=\sqrt{\mathcal{F}_t(y_0,y_0|1)}>0$. So, we have proved statement {\it (1)} of Lemma \ref{SPD lemma}. 

Due to {\it (1)}, the function $f(\varphi)=f_t(\varphi)=u\circ\mathscr{Z}^{(j)}_t(\epsilon,\varphi)$ is smooth in $t$. Since the coefficients $c_{j}(\nu)$ $(j>0)$ in {\rm(\ref{eigenfunc expansion near vert})} and their derivatives with respect to $\nu$ decay super-exponentially as $j\to+\infty$ or $\nu\to+\infty$, series {\rm(\ref{eigenfunc expansion near vert})} admit term-wise differentiation in $t,r,\varphi$ whence
$$\dot{\bf u}_j=O(r^{\frac{2\pi}{\beta}}), \qquad \frac{\partial\dot{\bf u}_j}{\partial r}=\Big(\sum_{\pm}c_\pm e^{\pm 2\pi i\varphi/\beta}\Big)r_j^{\frac{2\pi}{\beta}-1}+O(r^{\frac{2\pi}{\beta}-1+\delta}) \qquad (t=z_j)$$
(where $\delta>0$ and $c_\pm\in\mathbb{C}$) and
$$\dot{\bf u}_j=O(r^{\frac{2\pi}{\beta}}{\rm log}r), \quad \frac{\partial\dot{\bf u}_j}{\partial r}=O(r^{\frac{2\pi}{\beta}-1}{\rm log}r) \qquad (t=\beta_j).$$
Estimates 
$$\frac{\partial\xi_j}{\partial z_j}\circ\mathscr{Z}^{(j)}_t=O(r_j^{b_j/(b_j+1)}), \qquad \frac{\partial\xi_j}{\partial \beta_j}\circ\mathscr{Z}^{(j)}_t=O({\rm log}r_j)$$
follow from (\ref{conical coordinates}). Substituting the above formulas into the chain rule
$$\dot{u}\circ\mathscr{Z}^{(j)}_t={\bf u}_j-\frac{\partial{\bf u}}{\partial\xi_j}\,\cdot\dot{\xi_j}\circ\mathscr{Z}^{(j)}_t-\frac{\partial{\bf u}}{\partial\overline{\xi_j}}\,\cdot\overline{\dot{\xi_j}\circ\mathscr{Z}^{(j)}_t},$$
one arrives at (\ref{dot u near vertex est}). Thus, we have proved statement {\it (2)} of Lemma \ref{SPD lemma}.

Now, consider the case $\lambda_j(t)=\dots=\lambda_{j+m-1}(t)$ for all $t\in(-t_0,t_0)$; then $\mathcal{F}_t(\cdot,\cdot|\mathcal{E})=\mathcal{E}(\lambda_j)\mathcal{F}_t(\cdot,\cdot|1)$, the equation $(\Delta_t-\lambda_j(t))\mathcal{F}_t(\cdot,\cdot|1)=0$ holds outside vertices, and statement {\it (2)} of Lemma \ref{SPD lemma} is valid for $u=\mathcal{F}_t(\cdot,\cdot|1)$. Thus, $\lambda_j$ is differentiable in $t$ and the formula 
\begin{equation}
\label{gen E formula var eig}
\partial_t\Big(\sum_{k=0}^{m-1}\mathcal{E}(\lambda_{j+k})\Big)=\mathscr{I}\,{\rm p.v.}\int\limits_X\dot{\phi}_t(x)\mathcal{F}_t(x,x|\lambda\mathcal{E}'(\lambda))dS_t
\end{equation}
is obtained by repeating the reasoning leading to (\ref{variation of eigenvalues conic}) in the multiplicity one case. Thus, formula (\ref{gen E formula var eig}) is valid for all $t$ except the discrete set of parameters for which the multiplicities of $\lambda_j(t)$ change. 

It remains to prove that the left-hand side of (\ref{gen E formula var eig}) is smooth in $t$ (including the above exceptional values). Introduce the basis $v_j,\dots,v_{j+m-1}$ of solutions to $(\Delta_0-\lambda_{j+k}(0))v_{j+k}=0$ orthonormal in $L_2(X(\epsilon);m_0)$ (such a basis can be constructed by ortonormalization of $u_j,\dots,u_{j+m-1}$ for sufficiently small $\epsilon$). Denote $M_{kl}(t)=(u_{j+k}(\cdot,t),v_{j+l})_{L_{2}(X(\epsilon);m_0)}$ and introduce the matrix $D(t)={\rm diag}(\lambda_{j}(t),\dots,\lambda_{j+m-1}(t))$. Let $M=UA$ be a polar decomposition of $M$, where $A$ is positive hermitian, $A=\sqrt{M^*M}$, and $U$ is unitary. Note that $M(0)$ is close to the identity matrix for small $\epsilon$; hence, one can assume that $M(0)$ is invertible. We have
\begin{align}
\label{polar decomp different}
\begin{split}
\int\limits_{X(\epsilon)}\int\limits_{X(\epsilon)}\mathcal{F}_t(x,y|\mathcal{E})v_{j+k}(x)v_{j+l}(y)dS(x)dS(y)=\\=(M^*\mathcal{E}(D)M)_{k,l}=(AU^*\mathcal{E}(D)UA)_{k,l},
\end{split}
\end{align}
where the left-hand side is smooth in $t$. Put $\mathcal{E}=1$, then (\ref{polar decomp different}) implies that the matrix $A^2=M^*M$ is smooth in $t$. Therefore, $A^2$ is invertible and $A^{-2}$, $A$, $A^{-1}$ are smooth for $t$ close to zero. Now, from (\ref{polar decomp different}) it follows that
\begin{align}
\label{polar decomp different 1}
\begin{split}
A^{-1}_{pk}(t)\int\limits_{X(\epsilon)}\int\limits_{X(\epsilon)}\mathcal{F}_t(x,y|\mathcal{E})v_{j+k}(x)v_{j+l}(y)dS(x)dS(y)A^{-1}_{lp}(t)=\\={\rm Tr}(U^*\mathcal{E}(D)U)={\rm Tr}\mathcal{E}(D(t))=\sum_{k=0}^{m-1}\mathcal{E}(\lambda_{j+k}(t)).
\end{split}
\end{align}
is smooth in $t$. In conclusion, it worth noting that the operator $\mathscr{J}$ as well as principal values of the integrals do not appear in the above reasoning for the smooth metrics case considered in Section \ref{sec smooth Polyakov} since $\dot{\phi}(x)$ is smooth in $(x,t)$ in this case. \qed


\begin{thebibliography}{99}
	
	
	\bibitem{AS1} 
	\newblock {E. Aurell, P. Salomonson.} 
	\newblock{\em On Functional Determinants of Laplacians in Polygons and Simplicial Complexes.}
	\newblock{\em Communications in Mathematical Physics.} 165, 233-259 (1994)
	
	
	\bibitem{AS2}{E. Aurell, P. Salomonson.} 
	\newblock{\em Further results on Functional Determinants of Laplacians in Simplicial Complexes}
	\newblock{arXiv:hep-th/9405140v1}, May 1994

	\bibitem{BFK} 
	\newblock{Burghelea, D.; Friedlander, L.; Kappeler, T.}
	\newblock{\em Mayer-Vietoris type formula for determinants of elliptic differential operators.}
	\newblock{\em J. Funct. Anal.} 107 (1992), no. 1, 34–65.
	
	\bibitem{TC} 
	\newblock{T. Can.}
	\newblock {\em Notes on determinant formula for polyhedra.}
	\newblock{\em  Unpublished manuscript.} Private communication. May 2017.
	
	\bibitem{Carslaw}
	\newblock{ H.S. Carslaw.}
	\newblock {\em The Green's Function for a Wedge of any Angle, and Other Problems in the Conduction of Heat}.
	\newblock {\em Proc. Lond. Math. Soc.} 2(8) (1910): 365--374. DOI: \url{https://doi.org/10.1112/plms/s2-8.1.365}
	
	\bibitem{Dowker1}
	\newblock{J.S. Dowker.}
	\newblock {\em Quantum field theory on a cone.}
	\newblock {\em J. Phys. A: Math.} 10(1) (1977): 115--124. DOI: \url{https://doi.org/10.1088/0305-4470/10/1/023}
	
	\bibitem{Dowker}
	\newblock{J. S. Dowker.}
	\newblock{\em Effective action in spherical domains.} 
	\newblock{\em 
Commun. Math. Ph.} 162 (1994) 633-647
	
	\bibitem{Fay2}
	\newblock{John Fay.}
	\newblock {\em Kernel functions, analytic torsion, and moduli spaces.}
	\newblock {\em Memoirs of the AMS} 464, Providence, Rhode Island (1992), 123 p. ISBN: 082182550X.
	
	\bibitem{Kalvin}
	\newblock{V. Kalvin.}
	\newblock{\em Polyakov-Alvarez type comparison formulas for determinants of Laplacians on Riemann surfaces with conical singularities.}
	\newblock{Journal of Functional Analysis.} 280 (2021) 108866
	
	
%
	
	
	
	\bibitem{Klevtsov}
	\newblock{S. Klevtsov.}
	\newblock{Lowest Landau level on a cone and zeta determinants}
	\newblock{\em Journal of Physics A: Mathematical and Theoretical.} 50 (2017), 234003
	
	\bibitem{Kok}
	\newblock{A. Kokotov.}
	\newblock{\em Polyhedral surfaces and determinant of Laplacian.}
	\newblock{Proceedings of AMS.} Vol. 141, n. 2, 2013, p. 725-735
	
	\bibitem{KokKorWZ}
	\newblock{A. Kokotov, D. Korikov.} 
	\newblock{\em On a polygon version of Wiegmann-Zabrodin formula.}
	\newblock{arXiv:2503.13718}
	   
	\bibitem{KokKorJPhA}
\newblock{A. Kokotov, D. Korotkin.}
\newblock{\em Bergman tau-function:from random matrices and Frobenius manifolds to spaces of quadratic differentials.}
\newblock{\em Journal of Physics A: Mathematical and General.} 39 (2006) 8997-9013
	
	
	\bibitem{KokKor}
	\newblock{A. Kokotov, D. Korotkin.}
	\newblock {\em Tau-functions on spaces of Abelian differentials and higher genus generalizations of Ray-Singer formula}.
	\newblock {\em J. Differential Geom.} 82 (2004), 35--100.
	
	
	\bibitem{Mazzeo} 
\newblock{R. Mazzeo, J. Rowlett.}
\newblock{\em A heat trace anomaly on polygons.}
\newblock{\em Math. Proc. Cambridge Philos. Soc.} 159 (2015), no. 2, 303–319.

	
	\bibitem{MNP}
	\newblock{V. Mazy'a, S. Nazarov, B. Plamenevskii.}
	\newblock{\em Asymptotic Theory of Elliptic Boundary Value Problems in Singularly Perturbed Domains.} Springer, 2000
	
	\bibitem{NP}
	\newblock{S. Nazarov, B. Plamenevskii.}
	\newblock {Elliptic Problems in Domains with Piecewise Smooth Boundaries.} De Gruyter Expositions in Mathematics, 13 (1994), Berlin, New York (1994), 532 p.
	
	\bibitem{TakhMcInt}
	\newblock{A. McIntyre, L. Takhtajan.}
\newblock{\em Holomorphic factorization of determinants of laplacians on Riemann surfaces and a higher genus generalization of Kronecker's first limit formula.}
\newblock{\em GAFA.} (2006) Volume 16, p. 1291–1323
		
	\bibitem{Sarnak} 
	\newblock{B. Osgood, R. Phillips, P. Sarnak.}
	\newblock{\em Extremals of Determinants of laplacians.}
	\newblock{\em  Journal of Functional Analysis.} 80, 148-211 (1988)
	
	
	\bibitem{Spreafico}
	\newblock{M. Spreafico.}
	\newblock{\em Zeta function and regularized determinant on a disk and on a cone.}
	\newblock{Journal of Geometry and Physics.}  54 (2005) 355-371
	
	\bibitem{Troyanov} 
	Marc Troyanov.
	\newblock {Les surfaces euclidiennes \`a singularit\'es coniques}.
	\newblock {\em L'Enseignement Math\'ematique} 32(2) (1986): 79--94. DOI: \url{https://doi.org/10.5169/seals-55079}

\end{thebibliography}
\end{document}